\documentclass[12pt]{article}
\usepackage{amssymb, amsmath, latexsym, a4}
\usepackage[latin1]{inputenc}
\usepackage{graphicx}
\usepackage[all]{xy}
\usepackage[usenames]{color}

\newenvironment{proof}[1][Proof]{\noindent\textbf{#1.} }{\ \rule{0.5em}{0.5em}}

\newtheorem{De}{Definition}[section]
\newtheorem{Th}[De]{Theorem}
\newtheorem{Pro}[De]{Proposition}
\newtheorem{Le}[De]{Lemma}
\newtheorem{Co}[De]{Corollary}
\newtheorem{Rem}[De]{Remark}
\newtheorem{Ex}[De]{Example}

\newcommand{\im}{\ensuremath{\mathsf{Im\,}}}
\renewcommand{\Im}{\sf Im}

\renewcommand{\ker}{\ensuremath{\mathsf{Ker\,}}}

\newcommand{\Lie}{\ensuremath{\mathsf{Lie}}}

\newcommand{\Lieh}{\ensuremath{\mathfrak{h}}}
\newcommand{\Lieg}{\ensuremath{\mathfrak{g}}}
\newcommand{\Lieq}{\ensuremath{\mathfrak{q}}}

\newcommand{\Liem}{\ensuremath{\mathfrak{m}}}
\newcommand{\Lien}{\ensuremath{\mathfrak{n}}}
\newcommand{\Lief}{\ensuremath{\mathfrak{f}}}
\newcommand{\Lies}{\ensuremath{\mathfrak{s}}}
\newcommand{\Lier}{\ensuremath{\mathfrak{r}}}
\newcommand{\Liet}{\ensuremath{\mathfrak{t}}}
\newcommand{\Liep}{\ensuremath{\mathfrak{p}}}
\newcommand{\Lb}{\ensuremath{\mathsf{Leib}}}

\newcommand{\ze}{{\cal Z}}
\newcommand{\Leib}{\ensuremath{\mathsf{Leib}}}

\newbox\pullbackbox
\setbox\pullbackbox=\hbox{\xy 0;<1mm,0mm>: \POS(4,0)\ar@{-} (0,0) \ar@{-} (4,4)
\endxy}

\newdir{>>}{{}*!/3.5pt/:(1,-.2)@^{>}*!/3.5pt/:(1,+.2)@_{>}*!/7pt/:(1,-.2)@^{>}*!/7pt/:(1,+.2)@_{>}}
\newdir{ >>}{{}*!/8pt/@{|}*!/3.5pt/:(1,-.2)@^{>}*!/3.5pt/:(1,+.2)@_{>}}
\newdir{ |>}{{}*!/-3.5pt/@{|}*!/-8pt/:(1,-.2)@^{>}*!/-8pt/:(1,+.2)@_{>}}
\newdir{ >}{{}*!/-8pt/@{>}}
\newdir{>}{{}*:(1,-.2)@^{>}*:(1,+.2)@_{>}}
\newdir{<}{{}*:(1,+.2)@^{<}*:(1,-.2)@_{<}}

\newcommand{\n}{{\frak n}}
\newcommand{\g}{\frak g}
\newcommand{\q}{\frak q}

\begin{document}

\centerline{\bf  THE c-NILPOTENT SCHUR $\Lie$-MULTIPLIER OF LEIBNIZ ALGEBRAS}

\bigskip
\bigskip
\centerline{\bf G. R. Biyogmam$^1$ and J. M. Casas$^2$}

\bigskip
\centerline{$^1$ Department of Mathematics, Georgia College \& State University}
\centerline{Campus Box 17 Milledgeville, GA 31061-0490}
\centerline{ {E-mail address}: guy.biyogmam@gcsu.edu}
\bigskip

\centerline{$^2$Dpto. Matemática Aplicada I, Universidad de Vigo,  E. E. Forestal}
\centerline{Campus Universitario A Xunqueira, 36005 Pontevedra, Spain}
\centerline{ {E-mail address}: jmcasas@uvigo.es}
\bigskip

\date{}

\bigskip \bigskip \bigskip

{\bf Abstract:} We introduce the notion of $c$-nilpotent Schur $\Lie$-multiplier of  Leibniz algebras. We obtain exact sequences and formulas of the dimensions of the underlying vector spaces relating the $c$-nilpotent Schur $\Lie$-multiplier of a Leibniz algebra $\Lieq$ and its quotient by a two-sided ideal. These tools are used to characterize $\Lie$-nilpotency and $c$-$\Lie$-stem covers of Leibniz algebras. We prove the existence of $c$-$\Lie$-stem covers for finite dimensional Leibniz algebras and the non existence of $c$-covering on certain $\Lie$-nilpotent Leibniz algebras  with non trivial $c$-nilpotent Schur $\Lie$-multiplier, and we provide characterizations of  $c$-$\Lie$-capability of Leibniz algebras by means of both their $c$-$\Lie$-characteristic ideal and $c$-nilpotent Schur $\Lie$-multiplier.
\bigskip

\bigskip

\bigskip


\section{Introduction}

The general theory of central extensions relative to a chosen subcategory of a base category was introduced in \cite{JK}, where the simultaneous categorical and Galois theoretic approach is based on, and generalizes, the work of the Fr\"{o}hlich school \cite{Fro, FC, Lue} which focused in varieties of $\Omega$-groups,  was considered in the context of semi-abelian categories \cite{JMT} relative to a Birkhoff subcategory in \cite{CVDL, EVDL}. Examples like groups vs. abelian groups, Lie algebras vs. vector spaces are absolute, meaning that they fit in the general theory when the considered Birkhoff subcategory is the subcategory of all abelian objects. A non-absolute example is the category of Leibniz algebras together with the Birkhoff subcategory of Lie algebras. The general theory provides the notions of relative central extension and relative commutator with respect to the Liezation functor $(-)_{\Lie} : \Leib \to \Lie$ which assigns to a Leibniz algebra $\Lieg$ the Lie algebra ${\Lieg_\Lie} = \Lieg / \Lieg^{\rm ann}$, where $ \Lieg^{\rm ann} = \langle \{ [x,x] : x \in \Lieg \} \rangle$.

In the recent papers  \cite{BC, CKh, CI} authors approached the relative theory of  Leibniz algebras with respect to the Liezation functor,  yielding to the introduction of new notions of central extensions, capability, nilpotency, stem cover, isoclinism and Schur multiplier relative to the Liezation functor, the so-called $\Lie$-central extensions, $\Lie$-capability, $\Lie$-nilpotency, $\Lie$-stem cover, $\Lie$-isoclinism and Schur $\Lie$-multiplier.

It is well-known the interplay between the Schur multiplier \cite{Sch} and other mathematical concepts like projective representations, central extensions, efficient presentations or homology. The Schur multiplier was generalized by Baer \cite{Ba1, Ba2, Ba3} to any variety of nilpotent groups. When ${\cal V}$ is the variety of nilpotent Lie algebras of class at most $c \geq 1$, then the Baer-invariant $M^{(c)}(L) = \frac{R \cap \gamma_{c+1}(F)}{\gamma_{c+1}(F,R)}$  (that is, the quotient doesn't depend on the chosen free presentation \cite{EVDL}), where $0 \to R \to F \to L \to 0$ is a free presentation of the Lie algebra $L$ and $\gamma_{c+1}(-)$ are the $(c+1)$-terms of the lower central series, was introduced in \cite{SAE} and later developed, among others, in \cite{Ar, Ar1, AR, RiA, RS, SA,SR}.

Our goal in this paper is to introduce the relative notion of $c$-nilpotent Schur multiplier, the Baer-invariant called $c$-nilpotent Schur $\Lie$-multiplier,  ${\cal M}_{\Lie}^{(c)}(\Lieq) = \frac{\Lier \cap \gamma_{c+1}^{\Lie}(\Lief)}{\gamma_{c+1}^{\Lie}(\Lief,\Lier)}$, where $\gamma_{c+1}^{\Lie}(-)$ are the $(c+1)$-terms of the $\Lie$-lower central series of $\Lieq$ \cite{CKh} and $0 \to \Lier \to \Lief \to \Lieq \to 0$ is a free presentation. Then we apply it to characterize $\Lie$-nilpotency of  Leibniz algebras and study $c$-$\Lie$-stem covers and $c$-$\Lie$ capabality of Leibniz algebras.

The paper is organized as follows: in section \ref{preliminaries}, we provide preliminary results on Leibniz algebras from \cite{CKh} which are needed through the paper. In section \ref{Schur multiplier}, we introduce the notion of $c$-nilpotent Schur $\Lie$-multiplier $(c\geq1)$ and we use it along with the Baer-invariant $\gamma_{c+1}^{\Lie^{ \ast}}({\Lieq }) = \frac{\gamma_{c+1}^{\Lie}(\Lief)}{\gamma_{c+1}^{\Lie}( \Lief, \Lier)}$, where $0 \to \Lier \to \Lief \to \Lieq \to 0$ is a free presentation, to characterize $\Lie$-nilpotency of class $c$ for a given Leibniz algebra. In section \ref{dimension}, we determine exact sequences providing several results on the dimension of the $c$-nilpotent Schur $\Lie$ -multiplier in the case when the Leibniz algebra is finite dimensional.  In section \ref{stem}, we use $c$-nilpotent Schur $\Lie$-multiplier to provide a characterization of $c$-$\Lie$-stem extensions, and prove the existence of $c$-$\Lie$-stem covers for finite dimensional Leibniz algebras, and the non existence of $c$-$\Lie$-covering on certain nilpotent Leibniz algebras  with non trivial $c$-nilpotent Schur $\Lie$-multiplier. Finally, in section \ref{capability} we  study $c$-$\Lie$-capability. In particular we provide characterizations of  $c$-$\Lie$-capability of Leibniz algebras by means of both their $c$-$\Lie$-characteristic ideal and $c$-nilpotent Schur $\Lie$-multiplier.


\section{Preliminary results on Leibniz algebras} \label{preliminaries}
We fix $\mathbb{K}$ as a ground field such that $\frac{1}{2} \in \mathbb{K}$. All vector spaces and tensor products are considered over $\mathbb{K}$.

\subsection{Background on Leibniz algebras}

A \emph{Leibniz algebra} \cite{Lo 1, Lo 2, LP} is a $\mathbb{K}$-vector space $\Lieq$  equipped with a linear map $[-,-] : \Lieq \otimes \Lieq \to \Lieq$, usually called the \emph{Leibniz bracket} of $\Lieq$,  satisfying the \emph{Leibniz identity}:
\[
 [x,[y,z]]= [[x,y],z]-[[x,z],y], \quad x, y, z \in \Lieq.
\]
Leibniz algebras  constitute a variety of $\Omega$-groups \cite{Hig}, hence it is a semi-abelian variety \cite{CVDL, JMT} denoted by {\Leib}, whose morphisms are linear maps that preserve the Leibniz bracket.

 A subalgebra ${\Lieh}$ of a Leibniz algebra ${\Lieq}$ is said to be \emph{left (resp. right) ideal} of ${\Lieq}$ if $ [h,q]\in {\Lieh}$  (resp.  $ [q,h]\in {\Lieh}$), for all $h \in {\Lieh}$, $q \in {\Lieq}$. If ${\Lieh}$ is both
left and right ideal, then ${\Lieh}$ is called \emph{two-sided ideal} of ${\Lieq}$. In this case $\Lieq/\Lieh$ naturally inherits a Leibniz algebra structure.

For a Leibniz algebra ${\Lieq}$, we denote by ${\Lieq}^{\rm ann}$ the subspace of ${\Lieq}$ spanned by all elements of the form $[x,x], x \in \Lieq$.

Given a Leibniz algebra $\Lieq$, it is clear that the quotient ${\Lieq}_ {_{\rm Lie}}=\Lieq/{\Lieq}^{\rm ann}$ is a Lie algebra. This defines the so-called  \emph{Liezation functor} $(-)_{\Lie} : {\Leib} \to {\Lie}$, which assigns to a Leibniz algebra $\Lieq$ the Lie algebra ${\Lieq}_{_{\rm Lie}}$. Moreover, the canonical surjective homomorphism  ${\Lieq} \twoheadrightarrow {\Lieq}_ {_{\rm Lie}}$ is universal among all homomorphisms from $\Lieq$ to a Lie algebra, implying that the Liezation functor is left adjoint to the inclusion functor ${\Lie} \hookrightarrow {\Leib}$.

Since $\Lie$ is a subvariety of $\Leib$, then it is a Birkhoff subcategory of $\Leib$.

 For a Leibniz algebra {\Lieq} and two-sided ideals  ${\Liem}$ and ${\Lien}$ of  ${\Lieq}$, the \Lie-{\it centralizer} of ${\Liem}$ and ${\Lien}$ over  ${\Lieq}$ is
\[
C_{\Lieq}^{\Lie}({\Liem} , {\Lien}) = \{q \in {\Lieq} \mid  \; [q, m] + [m,q] \in {\Lien}, \; \text{for all} \;
m \in {\Liem} \} \; .
\]

 The \Lie-{\it commutator} $[\Liem,\Lien]_{\Lie}$ is the subspace  of $\Lieq$ spanned by all elements of the form $[m,n]+[n,m]$, $m \in \Liem$, $n \in \Lien$.

 In particular, the two-sided ideal $C_{\Lieq}^{\Lie}(\Lieq , 0)$ is the \emph{$\Lie$-center} of the Leibniz algebra $\Lieq$ and it will be denoted by $Z_{\Lie}(\Lieq)$, that is,

\[
Z_{\Lie}(\Lieq) =  \{ z\in \Lieq\,|\,\text{$[q,z]+[z,q]=0$ for all $q\in \Lieq$}\}.
\]

\begin{Pro} \cite[Example 1.9]{CVDL09}
Given an extension of Leibniz algebras $f : \Lieg \twoheadrightarrow \Lieq$ with ${\Lien} =  \ker(f)$, the following conditions are equivalent:
\begin{enumerate}
\item[$(a)$] $f : \Lieg \twoheadrightarrow \Lieq$ is $\Lie$-central;
\item[$(b)$] $\Lien \subseteq Z_{\Lie}(\Lieg)$;
\item[$(c)$] $[\Lien,\Lieg]_{\Lie} = 0$.
\end{enumerate}
\end{Pro}

\subsection{$\Lie$-nilpotent Leibniz algebras} \label{nil}

\begin{De} \cite[Definition 5]{CKh}
Let ${\Liem}$ be a two-sided ideal of a Leibniz algebra ${\Lieq}$. A series from ${\Liem}$ to ${\Lieq}$ is a finite sequence of two-sided ideals ${\Liem}_i$, $0 \leq i \leq k$, of ${\Lieq}$ such that
\[
{\Liem}= {\Liem}_0 \trianglelefteq {\Liem}_1 \trianglelefteq \dots \trianglelefteq {\Liem}_{k-1} \trianglelefteq
{\Liem}_k ={\Lieq} \; .
\]
$k$ is called the length of this series.

   A series from $\Liem$ to $\Lieq$ of length $k$ is said to be $\Lie$-central (resp. $\Lie$-abelian) if $[{\Liem}_i, {\Lieq}]_{\Lie} \subseteq {\Liem}_{i-1}$,  or equivalently ${\Liem}_i/{\Liem}_{i-1} \subseteq Z_{\Lie}({\Lieq}/{\Liem}_{i-1})$ (resp. if $[\Liem_i,\Liem_i]_{\Lie} \subseteq \Liem_{i-1}$, or equivalently $[\Liem_i/\Liem_{i-1}, \Liem_i/\Liem_{i-1}]_{\Lie} = 0$) for $1\leq i\leq k$ .

A series from $0$ to ${\Lieq}$ is called a series of the Leibniz algebra ${\Lieq}$.
\end{De}

\begin{De} \cite[Definition 11]{CKh}
Let ${\Lien}$ be a two-sided ideal of a Leibniz algebra $\Lieq$. The lower $\Lie$-central series of $\Lieq$ relative to ${\Lien}$ is the sequence
\[
\cdots \trianglelefteq {\gamma_i^{\Lie}(\Lieq,\Lien)} \trianglelefteq \cdots \trianglelefteq \gamma_2^{\Lie}(\Lieq,\Lien)  \trianglelefteq {\gamma_1^{\Lie}(\Lieq,\Lien)}
\]
of two-sided ideals of $\Lieq$ defined inductively by
\[
{\gamma_1^{\Lie}(\Lieq,\Lien)} = {\Lien} \quad \text{and} \quad \gamma_i^{\Lie}(\Lieq,\Lien) =[{\gamma_{i-1}^{\Lie}(\Lieq,\Lien)},{\Lieq}]_{\Lie}, \quad   i \geq 2.
\]
The Leibniz algebra $\Lieq$ is said to be $\Lie$-nilpotent relative to $\Lien$ of class $c$ if\ $\gamma_{c+1}^{\Lie}(\Lieq, \Lien) = 0$ and $\gamma_c^{\Lie}(\Lieq, \Lien) \neq 0$.
\end{De}

Note that $[{\gamma_i^{\Lie}(\Lieq,\Lien)}/{\gamma_{i+1}^{\Lie}(\Lieq,\Lien)},{\gamma_i^{\Lie}(\Lieq,\Lien)}/{\gamma_{i+1}^{\Lie}(\Lieq,\Lien)}]_{\Lie}={0}$. When ${\Lien} = {\Lieq}$ we obtain  Definition 9 in \cite{CKh} and we will use the notation $\gamma_i^{\Lie}(\Lieq)$ instead of $\gamma_i^{\Lie}(\Lieq,\Lieq), 1 \leq i \leq n$.
If $\varphi : \Lieg \to \Lieq$ is a homomorphism of Leibniz algebras such that $\varphi({\Liem}) \subseteq {\Lien}$, where ${\Liem}$ is a two-sided ideal of ${\Lieg}$ and ${\Lien}$ a two-sided ideal of ${\Lieq}$, then $\varphi(\gamma_{i}^{\Lie}({\Lieg}, {\Liem})) \subseteq \gamma_{i}^{\Lie}({\Lieq}, {\Lien}), i \geq 1$.

\begin{De} \cite[Definition 10]{CKh}
The upper $\Lie$-central series of a Leibniz algebra ${\Lieq}$ is the sequence of
 two-sided ideals
 \[
{\ze}_0^{\Lie}({\Lieq}) \trianglelefteq {\ze}_1^{\Lie}({\Lieq}) \trianglelefteq \cdots \trianglelefteq {\ze}_i^{\Lie}({\Lieq}) \trianglelefteq \cdots
\]
 defined inductively by
\[
{\ze}_0^{\Lie}({\Lieq}) = 0 \quad \text{and} \quad
 {\ze}_{i}^{\Lie}({\Lieq}) = C_{\Lieq}^{\Lie}({\Lieq},{\ze}_{i-1}^{\Lie}({\Lieq})) , \  i \geq 1 .
 \]
\end{De}

\begin{Th} \cite[Theorem 4]{CKh}
  A Leibniz algebra ${\Lieq}$ is $\Lie$-nilpotent with class of $\Lie$-nilpotency $k$ if and only if
${\ze}_k^{\Lie}({\Lieq}) = {\Lieq}$ and ${\ze}_{k-1}^{\Lie}({\Lieq}) \ne {\Lieq}$.
\end{Th}

\section{The c-nilpotent Schur \Lie-multiplier} \label{Schur multiplier}

 In this section we introduce the Baer-invariants $c$-nilpotent Schur $\Lie$-multiplier and  $\gamma_{c+1}^{\Lie^{ \ast}}(-)$  of a Leibniz algebra. Then we use them to characterize $\Lie$-nilpotency of Leibniz algebras.

Let $0 \to \Lier \to \Lief \stackrel{\rho}\to \Lieq \to 0$ be a free presentation of a Leibniz algebra $\Lieq$. We call $c$-nilpotent Schur $\Lie$-multiplier  of $\Lieq \ (c \geq 1)$ to the term
\begin{equation} \label{Schur formula}
{\cal M}_{\Lie}^{(c)} (\Lieq):= \frac{\Lier \cap {\gamma_{c+1}^{\Lie}(\Lief)}}{{\gamma_{c+1}^{\Lie}(\Lief,\Lier)}}
\end{equation}

Since {\Lb} is a category of interest (see \cite{CDL}), hence is a category of $\Omega$-groups, and following Proposition 4.3.2 in \cite{Huq} we can conclude that the collection of all nilpotent objects of class $\leq c$ in {\Lb} form a variety. Now following  \cite[Proposition 7.8]{EVDL}, it can be showed that ${\cal M}_{\Lie}^{(c)} (\Lieq)$ and $\gamma_{c+1}^{\Lie^{ \ast}}({\Lieq })= \frac{\gamma_{c+1}^{\Lie}({\Lief})}{\gamma_{c+1}^{\Lie}({\Lief},{\Lier})}$ are Baer-invariants, which means that their definitions do not depend on the choice of the  free presentation.

\begin{Rem}
${\cal M}_{\Lie}^{(1)} (\Lieq)$ is the Schur $\Lie$-multiplier of a Leibniz algebra $\Lieq$  (see \cite{BC, CI}).
\end{Rem}

\begin{De}
An exact sequence of Leibniz algebras $0 \to \Lien \to \Lieg \stackrel{\pi} \to \Lieq \to 0$ is said to be a $c$-$\Lie$-central extension if $\gamma_{c+1}^{\Lie}(\Lieg, \Lien)=0$.
\end{De}

\begin{Ex}\ \label{example Lie stem}
\begin{enumerate}

\item[(a)]
Let $\Lieg$ and $\Lieq$ be  four and two-dimensional Leibniz algebra with respective $\mathbb{K}$-linear bases $\{a_1, a_2, a_3, a_4\}$  and $\{e_1,e_2 \}$, with the Leibniz brackets given respectively by $[a_1,a_2] = a_3,  [a_2,a_2] =a_4, [a_1,a_3]= a_4$ and $[e_2,e_2] =  e_1$ and zero elsewhere ($\Lieg$ is a $\Lie$-nilpotent Leibniz algebra of class 3 and $\Lieq$ is a Lie-nilpotent Leibniz algebra of class 2 \cite{CI}). Then the surjective homomorphism of Leibniz algebras $f : \Lieg \twoheadrightarrow \Lieq$ defined by $f(a_1)=e_2$, $f(a_2)=e_1, f(a_3)=0$ and $f(a_4)=0$ is a $2$-$\Lie$-central extension.

\item[(b)]
Let $\Lieg$ and $\Lieq$ be  four and two-dimensional Leibniz algebra with respective $\mathbb{K}$-linear bases $\{a_1, a_2, a_3, a_4\}$  and $\{e_1,e_2 \}$, with the Leibniz brackets given respectively by $[a_1,a_1] = a_2,  [a_1,a_2] =a_3, [a_1,a_3]= a_4$ and $[e_2,e_2] =  e_1$ and zero elsewhere ($\Lieg$ is a $\Lie$-nilpotent Leibniz algebra of class 4 and $\Lieq$ is a $\Lie$-nilpotent Leibniz algebra of class 2 \cite{CI}). Then the surjective homomorphism of Leibniz algebras $f : \Lieg \twoheadrightarrow \Lieq$ defined by $f(a_1)=e_2$, $f(a_2)=e_1, f(a_3)=0$ and $f(a_4)=0$ is a $2$-$\Lie$-central extension.
\end{enumerate}
\end{Ex}

\begin{Pro} \label{inclusion}
The exact sequence of Leibniz algebras $0 \to \Lien \to \Lieg \stackrel{\pi} \to \Lieq \to 0$ is a  $c$-$\Lie$-central extension if and only if $\Lien \subseteq {\ze}_c^{\Lie}(\Lieg)$
\end{Pro}
\begin{proof}
Assume that the exact sequence of Leibniz algebras $0 \to \Lien \to \Lieg \stackrel{\pi} \to \Lieq \to 0$ is a  $c$-$\Lie$-central extension, then $\gamma_{c+1}^{\Lie}(\Lieg, \Lien) =0$. Let $t\in\Lien$ with $t\notin {\ze}_c^{\Lie}(\Lieg).$ Then $t_1:=[g_1,t]+[t,g_1]\notin {\ze}_{c-1}^{\Lie}(\Lieg)$ for some $g_1\in\Lieg.$ So $t_2:=[g_2,t_1]+[t_1,g_2]\notin {\ze}_{c-2}^{\Lie}(\Lieg)$ for some $g_2\in\Lieg.$ Inductively, we have $t_c:=[g_c,t_{c-1}]+[t_{c-1},g_c]\notin {\ze}_{0}^{\Lie}(\Lieg)=0$ for some $g_c\in\Lieg.$ So $t_c\neq0.$ On the other hand, we have  $t_c\in\gamma_{c+1}^{\Lie}(\Lieg, \Lien) =0 $ since $t\in\Lien=\gamma_{1}^{\Lie}(\Lieg, \Lien).$ A contradiction.

 Conversely, Assume that $\Lien \subseteq {\ze}_c^{\Lie}(\Lieg)$ and let $z\in\gamma_{c+1}^{\Lie}(\Lieg, \Lien).$ Then  $z:=[z_1,g_1]+[g_1,z_1]$ for some $z_1\in\gamma_{c}^{\Lie}(\Lieg, \Lien)$ and $g_1\in\Lieg.$ So  $z_1:=[z_2,g_2]+[g_2,z_2]$ for some $z_2\in\gamma_{c-1}^{\Lie}(\Lieg, \Lien)$ and $g_2\in\Lieg.$   Inductively we have $z_{c+1}:=[z_c,g_c]+[g_c,z_c]$ for some $z_c\in\gamma_{1}^{\Lie}(\Lieg, \Lien)$ and $g_c\in\Lieg.$ Since $\gamma_{1}^{\Lie}(\Lieg, \Lien)=\Lien \subseteq {\ze}_c^{\Lie}(\Lieg),$ it follows that $z_{c}\in {\ze}_c^{\Lie}(\Lieg).$ Inductively, we have $z\in {\ze}_0^{\Lie}(\Lieg)=0.$ Thus $z=0.$ Therefore $\gamma_{c+1}^{\Lie}(\Lieg, \Lien)=0.$
\end{proof}

\begin{Ex}\ \label{example exact sequence}
\begin{enumerate}

\item[(a)] The short exact sequence $0 \to {\cal M}_{\Lie}^{(c)} (\Lieq) \to \gamma_{c+1}^{\Lie^{ \ast}}({\Lieq }) \overset{u} \to \gamma_{c+1}^{\Lie}({\Lieq }) \to 0$, where $u(x+\gamma_{c+1}^{\Lie}({\Lief, \Lier })) = \rho(x)$, is a c-{\Lie}-central extension.

\item[(b)] From the $3 \times 3$ Lemma \cite[Theorem 4.2.7]{BB}, the free presentation $0 \to \Lier \to \Lief \stackrel{\rho}\to \Lieq \to 0$  induces the exact sequence $ 0 \to \frac{\Lier}{\gamma_{c+1}^{\Lie}(\Lief, \Lier)} \to  \frac{\Lief}{\gamma_{c+1}^{\Lie}(\Lief, \Lier)} \stackrel{\overline{\rho}}\to \Lieq \to 0$, where $\overline{\rho}$ is the natural epimorphism induced by $\rho$. Moreover this exact sequence is a $c$-$\Lie$-central extension.
\end{enumerate}
\end{Ex}

\begin{Le} \label{lema 1}
Let $0 \to \Lier \to \Lief \stackrel{\rho} \to \Lieq \to 0$ be a free presentation of a Leibniz algebra $\Lieq$ and let $0 \to \Liem \to \Lieh \stackrel{\theta} \to \Liep \to 0$ be a $c$-$\Lie$-central extension of another Leibniz algebra $\Liep$. Then for each homomorphism $\alpha : \Lieq \to \Liep$, there exists a homomorphism $\beta : \frac{\Lief}{\gamma_{c+1}^{\Lie}(\Lief, \Lier)} \to \Lieh$ such that $\beta \left( \frac{\Lier}{\gamma_{c+1}^{\Lie}(\Lief, \Lier)} \right) \subseteq \Liem$ and the following diagram is commutative:
\begin{equation} \label{diagram lemma}
\xymatrix{
0 \ar[r] & \frac{\Lier}{\gamma_{c+1}^{\Lie}(\Lief, \Lier)} \ar[r] \ar[d]^{\beta_{\mid}} & \frac{\Lief}{\gamma_{c+1}^{\Lie}(\Lief, \Lier)} \ar[r]^{\overline{\rho}} \ar[d]^{\beta} & \Lieq \ar[r] \ar[d]^{\alpha} & 0\\
0 \ar[r] & \Liem \ar[r] & \Lieh \ar[r]^{\theta} & \Liep \ar[r] & 0
}
\end{equation}
\end{Le}
\begin{proof}
Since $\Lief$ is a free Leibniz algebra, then there exists $\omega : \Lief \to \Lieh$ such that $\theta \circ \omega = \alpha \circ \rho$.

On the other hand, $\theta ( \omega(\Lier)) = \alpha(\rho(\Lier))=0$, hence $\omega(\Lier) \subseteq \Liem$, which implies that $\omega(\gamma_{c+1}^{\Lie}(\Lief, \Lier)) = 0$, hence $\omega$ induces $\beta : \frac{\Lief}{\gamma_{c+1}^{\Lie}(\Lief, \Lier)} \to \Lieh$ such that $\alpha \circ \overline{\rho} = \theta \circ \beta$, and for any $r \in \Lier$,  $\beta(r+\gamma_{c+1}^{\Lie}(\Lief, \Lier)) = \omega(r) \in \Liem$.
\end{proof}

\begin{Pro}
Let ${\Lieq}$ be a Leibniz algebra and $c \geq 1$. Then
\begin{enumerate}
\item[(a)] $\gamma_{c+1}^{\Lie^{ \ast}}({\Lieq})= 0$ if and only if {\Lieq} is $\Lie$-nilpotent of class $c$ and ${\cal M}_{\Lie}^{(c)} (\Lieq)=0$.
\item[(b)]  If $\gamma_{c+1}^{\Lie^{ \ast}}({\Lieq })=0$, then $\gamma_{c+1}^{\Lie^{ \ast}}({\Lieq}/{\Lien})= 0$ for any two-sided ideal ${\Lien}$ of ${\Lieq}$.
\end{enumerate}
\end{Pro}
\begin{proof} {\it (a)} $\gamma_{c+1}^{\Lie^{ \ast}}({\Lieq })= 0$ implies that $\gamma^{\Lie}_{c+1}({\Lief}) \subseteq \gamma^{\Lie}_{c+1}({\Lief}, {\Lier})$, thus   $\gamma^{\Lie}_{c+1}({\Lieq})=[\gamma^{\Lie}_{c}({\Lieq}), {\Lieq}]_{\Lie} = \gamma^{\Lie}_{c+1}({\Lief})/{\Lier} \subseteq \gamma^{\Lie}_{c+1}({\Lief},{\Lier})/{\Lier} \subseteq 0$.

Moreover ${\cal M}_{\Lie}^{(c)} (\Lieq)= \frac{{\Lier} \cap  \gamma^{\Lie}_{c+1}({\Lief})}{\gamma^{\Lie}_{c+1}({\Lief},{\Lier})} \subseteq \frac{{\Lier} \cap  \gamma^{\Lie}_{c+1}({\Lief},{\Lier})}{\gamma^{\Lie}_{c+1}({\Lief},{\Lier})}=0$.

Conversely, if ${\cal M}_{\Lie}^{(c)} (\Lieq) = 0$, then ${\Lier} \cap \gamma^{\Lie}_{c+1}({\Lief}) \subseteq \gamma^{\Lie}_{c+1}({\Lief},{\Lier})$. Since $\gamma^{\Lie}_{c+1}({\Lieq})  =0$ implies that $\gamma^{\Lie}_{c+1} ({\Lief})/{\Lier} =0$, then $\gamma^{\Lie}_{c+1}({\Lief})  \subseteq {\Lier}$. Hence $\gamma^{\Lie}_{c+1}({\Lief})  \subseteq \gamma^{\Lie}_{c+1}({\Lief},{\Lier})$ and, consequently,  $\gamma_{c+1}^{\Lie^{ \ast}}({\Lieq }) = 0$.

{\it (b)} Let ${\Lien}$ be a two-sided ideal of ${\Lieq}$ and consider the free presentation $0 \to {\Lies} = \ker(\pi \circ \rho) \to {\Lief} \stackrel{\tau \circ \rho} \to {\Lieq}/{\Lien} \to 0$, where $\tau : {\Lieq} \twoheadrightarrow{\Lieq}/{\Lien}$ is the canonical projection. Since ${\Lier} \subseteq {\Lies}$, then
$\gamma_{c+1}^{\Lie^{ \ast}}({\Lieq})= 0$ implies that $\gamma^{\Lie}_{c+1}({\Lief}) \subseteq \gamma^{\Lie}_{c+1}({\Lief}, {\Lier}) \subseteq \gamma^{\Lie}_{c+1}({\Lief}, {\Lies})$ which completes the proof.
\end{proof}

\begin{De}
Let $\Lieq$ be a \Lie-nilpotent Leibniz algebra of class $c$. An extension of Leibniz algebras $0 \to \Lien \to \Lieg \stackrel{\pi}\to \Lieq \to 0$ is said to be of class $c$ if $\Lieg$ is nilpotent of class $c$.
\end{De}

\begin{Th}
A $c$-$\Lie$-central extension $0 \to \Lien \to \Lieg \stackrel{\pi}\to \Lieq \to 0$ of a class $c$ nilpotent Leibniz algebra $\Lieq$   is of class $c$ if and only if $\theta : {\cal M}_{\Lie}^{(c)} \left({\Lieg} \right) \to \Lien$  vanishes over $\ker(\tau)$, where $\tau : {\cal M}_{\Lie}^{(c)} \left( {\Lieq} \right)\to{\cal M}_{\Lie}^{(c)} \left( \Lieq/\gamma_c^{\Lie}(\Lieq) \right)$ is induced by the canonical projection $\Lieq \twoheadrightarrow  \Lieq/\gamma_c^{\Lie}(\Lieq)$.
\end{Th}
\begin{proof}
Consider the following diagrams:
\[
\xymatrix{& & 0 \ar[d] & 0 \ar[ld]  &   &      & & 0 \ar[d] & 0 \ar[ld]  &     \\
&  & {\Lier} \ar[d]\ar[ld] & &              &  &  & {\Lies} \ar[d]\ar[ld] & &  \\
0 \ar[r] & {\Lies} \ \ar[r] \ar[d] & {\Lief} \ar[d]^{\rho} \ar[rd]^{\pi \circ \rho} & &    &  0 \ar[r] & {\Liet} \ \ar[r] \ar[d] & {\Lief} \ar[d]^{\pi \circ \rho} \ar[rd]^{pr \circ \pi \circ \rho} & &   \\
0 \ar[r]& {\Lien} \ar[r] \ar[d]& {\Lieg} \ar[r]^{\pi} \ar[d]& {\Lieq} \ar[r] \ar[dr] & 0   &  0 \ar[r]& \gamma_c^{\Lie}(\Lieq) \ar[r] \ar[d]& {\Lieq} \ar[r]^{pr\ \ \ \ } \ar[d]& {\Lieq}/\gamma_c^{\Lie}(\Lieq) \ar[r] \ar[dr] & 0 \\
 & 0 & 0 &  & 0 & & 0 & 0 &  & 0
}
\]
where $\Lies = \ker(\pi \circ \rho)$ and $\Liet = \ker(pr \circ \pi \circ \rho)$, then $\theta : {\cal M}_{\Lie}^{(c)} \left( \Lieq \right) = \frac{\Lies \cap \gamma_{c+1}^{\Lie}(\Lief)}{\gamma_{c+1}^{\Lie}(\Lief, \Lies)} \to \Lien$, given by $\theta(x+\gamma_{c+1}^{\Lie}(\Lief, \Lies))= \rho(x)$, is well-defined and ${\ker}(\tau) \cong \frac{\gamma_{c+1}^{\Lie}(\Lief, \Liet)}{\gamma_{c+1}^{\Lie}(\Lief, \Lies)}$.

Assume that {\Lieg} is \Lie-nilpotent of class $c$ and consider $x+\gamma_{c+1}^{\Lie}(\Lief, \Lies) \in {\ker}(\tau)$. Then $\theta(x+ \gamma_{c+1}^{\Lie}(\Lief, \Lies)) = \rho(x) = 0$ since $\rho(x) \in \rho(\gamma_{c+1}^{\Lie}(\Lief, \Liet)) \subseteq \gamma_{c+1}^{\Lie}(\Lieg) +\gamma_{c+1}^{\Lie}(\Lieg, \Lien) = 0$. For the last inclusion, it is necessary to have in mind that $\pi \left(\rho({\Liet}) \right) \subseteq \gamma_{c}^{\Lie}(\Lieq) = \pi(\gamma_{c}^{\Lie}(\Lieg))$, and consequently $\rho({\Liet}) \subseteq \gamma_{c}^{\Lie}(\Lieg) +{\Lien}$.

Conversely, $\gamma_{c+1}^{\Lie}(\Lieg) = [\gamma_{c}^{\Lie}(\Lieg), {\Lieg}]_{\Lie} = [\rho(\gamma_{c}^{\Lie}(\Lief)), \rho({\Lief})]_{\Lie} \subseteq \rho (\gamma_{c+1}^{\Lie}(\Lief, \Liet))  =0$  since $\gamma_{c+1}^{\Lie}(\Lief, \Liet) \subseteq {\Lier}$ because $\theta$ vanishes over ${\ker}(\tau)$. For the last inclusion is necessary to have in mind that $\pi \left( \rho(\gamma_{c}^{\Lie}(\Lief)) \right) \subseteq \gamma_{c}^{\Lie}(\Lieq)$, hence $\gamma_{c}^{\Lie}(\Lief) \subseteq {\Liet}$.
\end{proof}

\begin{Pro}
Let $\Lieg$ be a $\Lie$-nilpotent Leibniz algebra of class $c$ and $f : \Lieg \twoheadrightarrow \Lieq$ be a surjective homomorphism of Leibniz algebras. If $\ker(f) \subseteq \gamma_c^{\Lie}(\Lieg)$ and ${\cal M}_{\Lie}^{(c)}(\Lieq) =0$, then $f$ is an isomorphism. In particular, if ${\cal M}_{\Lie}^{(c)}(\Lieg/\gamma_c^{\Lie}(\Lieg)) =0$, then ${\cal M}_{\Lie}^{(c)}(\Lieg) =0$.
\end{Pro}
\begin{proof}
Let $\Lien = \ker(f)$, then ${\cal M}_{\Lie}^{(c)}(\Lieg/\Lien) =0$.
From the exact sequence in Proposition \ref{exact sequences} {\it (a)} we have that $\Lien \cap \gamma_{c+1}^{\Lie}(\Lieg) \subseteq \gamma_{c+1}^{\Lie}(\Lieg, \Lien)$, then $\Lien \subseteq \gamma_{c+1}^{\Lie}(\Lieg, \Lien)$. Obviously $\supseteq$ is true, then  $\Lien = \gamma_{c+1}^{\Lie}(\Lieg, \Lien) \subseteq \gamma_{c+1}^{\Lie}(\Lieg) = 0$. Consequently, $f$ is an isomorphism.
\end{proof}

\section{On the dimension of the $c$-nilpotent Schur $\Lie$-multiplier} \label{dimension}

 This section is devoted to obtain some exact sequences involving the $c$-nilpotent Schur $\Lie$-multiplier and some formulas concerning dimensions of the corresponding underlying vector spaces.

\begin{Pro} \label{exact sequences}
Let $0 \to \Lier \to \Lief \stackrel{\rho} \to \Lieq \to 0$ be a free presentation of a Leibniz algebra $\Lieq$. Let $\Lien$ a two-sided ideal of $\Lieq$ and $\Lies$ a two-sided ideal of $\Lief$ such that $\Lien \cong \frac{\Lies}{\Lier}$ ($\Lies = \ker(\pi \circ \rho)$, where $\pi : \Lieq \twoheadrightarrow \frac{\Lieq}{\Lien}$ is the canonical projection). Then the following sequences are exact:
\begin{enumerate}
\item[(a)] $0 \to \frac{\Lier \cap \gamma_{c+1}^{\Lie}(\Lief, \Lies)}{\gamma_{c+1}^{\Lie}(\Lief, \Lier)} \to {\cal M}_{\Lie}^{(c)} \left(\Lieq \right) \to {\cal M}_{\Lie}^{(c)} \left( \frac{\Lieq}{\Lien} \right)  \to \frac{\Lien \cap \gamma_{c+1}^{\Lie}(\Lieq)}{\gamma_{c+1}^{\Lie}(\Lieq, \Lien)} \to 0.$

    \item[(b)] ${\cal M}_{\Lie}^{(c)} (\Lieq) \to {\cal M}_{\Lie}^{(c)} \left( \frac{\Lieq}{\Lien} \right) \to \frac{\Lien \cap \gamma_{c+1}^{\Lie}(\Lieq)}{\gamma_{c+1}^{\Lie}(\Lieq, \Lien)} \to \frac{\Lien}{\gamma_{c+1}^{\Lie}(\Lieq, \Lien)} \to \frac{\Lieq}{\gamma_{c+1}^{\Lie}(\Lieq)} \to \frac{\Lieq}{\Lien + \gamma_{c+1}^{\Lie}(\Lieq)} \to 0.$

    \item[(c)] $\Lien \otimes^c \Lieq_{\Lie} \to {\cal M}_{\Lie}^{(c)} (\Lieq) \to {\cal M}_{\Lie}^{(c)} \left( \frac{\Lieq}{\Lien} \right) \to \Lien \cap \gamma_{c+1}^{\Lie}(\Lieq) \to 0$,  provided that $0 \to \Lien \to \Lieq \to \frac{\Lieq}{\Lien} \to 0$ is a $c$-{\Lie}-central extension. Here $\Lien \otimes^c \Lieq = \Lien \otimes \underbrace{\Lieq \otimes \dots \otimes \Lieq}_{c\, {\rm times}}$.
\end{enumerate}
\end{Pro}
\begin{proof} {\it (a)} Consider the following diagram:
\begin{equation} \label{free present diagr}
\xymatrix{& & 0 \ar[d] & 0 \ar[ld]\\
&  & {\Lier} \ar[d]\ar[ld] \\
0 \ar[r] & {\Lies} \ \ar[r] \ar[d] & {\Lief} \ar[d]^{\rho} \ar[rd]^{\pi \circ \rho} \\
0 \ar[r]& {\Lien} \ar[r] \ar[d]& {\Lieq} \ar[r]^{\pi} \ar[d]& \frac{{\Lieq}}{\Lien} \ar[r] \ar[dr] & 0 \\
 & 0 & 0 &  & 0
}
\end{equation}
Define $\Pi : {\cal M}_{\Lie}^{(c)} (\Lieq)= \frac{\Lier \cap {\gamma_{c+1}^{\Lie}(\Lief)}}{{\gamma_{c+1}^{\Lie}(\Lief,\Lier)}} \to  \frac{\Lies \cap {\gamma_{c+1}^{\Lie}(\Lief)}}{{\gamma_{c+1}^{\Lie}(\Lief,\Lies)}} = {\cal M}_{\Lie}^{(c)} \left( \frac{\Lieq}{\Lien} \right)$ by $\Pi (r + {\gamma_{c+1}^{\Lie}(\Lief,\Lier)}) = r + {\gamma_{c+1}^{\Lie}(\Lief,\Lies)}$. Obviously $\ker(\Pi) = \frac{\Lier \cap \gamma_{c+1}^{\Lie}(\Lief, \Lies)}{\gamma_{c+1}^{\Lie}(\Lief, \Lier)}$ and ${\sf Coker}(\Pi) =  \frac{\Lien \cap \gamma_{c+1}^{\Lie}(\Lieq)}{\gamma_{c+1}^{\Lie}(\Lieq, \Lien)}$ thanks to the following commutative diagram:
\begin{equation}\label{diagram1}
\xymatrix{
 \Lier \cap \gamma_{c+1}^{\Lie}(\Lief, \Lies)\  \ar@{>->}[r]  \ar@{>->}[d] & \Lier \cap \gamma_{c+1}^{\Lie}(\Lief) \ar@{>>}[r] \ar@{>->}[d] & \frac{\Lier \cap \gamma_{c+1}^{\Lie}(\Lief)}{\Lier \cap \gamma_{c+1}^{\Lie}(\Lief, \Lies)} \ar@{>->}[d]\\
\gamma_{c+1}^{\Lie}(\Lief, \Lies)\  \ar@{>->}[r] \ar@{>>}[d] & \Lies \cap \gamma_{c+1}^{\Lie}(\Lief) \ar@{>>}[r] \ar@{>>}[d] & \frac{\Lies \cap \gamma_{c+1}^{\Lie}(\Lief)}{\gamma_{c+1}^{\Lie}(\Lief, \Lies)} \ar@{>>}[d]\\
 \gamma_{c+1}^{\Lie}(\Lieq, \Lien) \cong \frac{\gamma_{c+1}^{\Lie}(\Lief, \Lies)}{\Lier \cap \gamma_{c+1}^{\Lie}(\Lief, \Lies)} \  \ar@{>->}[r]& \Lien \cap  \gamma_{c+1}^{\Lie}(\Lieq)  \ar@{>>}[r] & \frac{\Lien \cap  \gamma_{c+1}^{\Lie}(\Lieq)}{\gamma_{c+1}^{\Lie}(\Lieq, \Lien) }
}
\end{equation}

{\it (b)} Combine statement {\it (a)} with the following diagram:
\[
\xymatrix{
 \Lien \cap \gamma_{c+1}^{\Lie}(\Lieq)\ \ar@{>->}[rr] \ar@{>->}[dd] & & \gamma_{c+1}^{\Lie}(\Lieq) \ar@{>>}[r] \ar@{>->}[dd] \ar@{>->}[dl]& \frac{ \gamma_{c+1}^{\Lie}(\Lieq)}{\Lien \cap \gamma_{c+1}^{\Lie}(\Lieq)} \ar@{>->}[dd]\\
 & \Lien + \gamma_{c+1}^{\Lie}(\Lieq)  \ar@{>>}[ddl]& & \\
\Lien\ \ar@{>->}[rr] \ar@{>>}[d] \ar@{>->}[ur]&  &\Lieq \ar@{>>}[r] \ar@{>>}[d] & \frac{\Lieq}{\Lien} \ar@{>>}[d]\\
 \frac{\Lien + \gamma_{c+1}^{\Lie}(\Lieq)}{\gamma_{c+1}^{\Lie}(\Lieq)}\ \ar@{>->}[rr]& & \frac{\Lieq}{\gamma_{c+1}^{\Lie}(\Lieq)} \ar@{>>}[r] & \frac{\Lieq}{\Lien + \gamma_{c+1}^{\Lie}(\Lieq)}
}
\]
and have in mind the isomorphisms $\frac{\Lien + \gamma_{c+1}^{\Lie}(\Lieq)}{\gamma_{c+1}^{\Lie}(\Lieq)} \cong \frac{\Lien}{\Lien \cap \gamma_{c+1}^{\Lie}(\Lieq)} \cong \frac{\Lien}{\gamma_{c+1}^{\Lie}(\Lieq, \Lien)}$.

{\it (c)}  If $\Lien$ is $c$-{\Lie}-central in $\Lieq$, then {\it (a)} provides the exact sequence
\begin{equation} \label{four term}
0 \to \frac{\gamma_{c+1}^{\Lie}(\Lief, \Lies)}{\gamma_{c+1}^{\Lie}(\Lief, \Lier)} \to {\cal M}_{\Lie}^{(c)} (\Lieq) \to {\cal M}_{\Lie}^{(c)} \left( \frac{\Lieq}{\Lien} \right) \to \Lien \cap \gamma_{c+1}^{\Lie}(\Lieq)\to 0
\end{equation}
Now consider  $\sigma : \Lien \otimes \Lieq_{\Lie} \otimes \stackrel{c} \dots \otimes \Lieq_{\Lie} \to \frac{\gamma_{c+1}^{\Lie}(\Lief, \Lies)}{\gamma_{c+1}^{\Lie}(\Lief, \Lier)}$ given by $\sigma(n \otimes \overline{q_1} \otimes \dots \otimes \overline{q_c} )= [[[s,f_1]_{\Lie},f_2]_{\Lie}, \dots, f_c]_{\Lie} + \gamma_{c+1}^{\Lie}(\Lief, \Lier)$, where $\rho(s) = n$ and $\rho(f_i) = q_i, 1 \leq i \leq c$. The condition $\gamma_{c+1}^{\Lie}(\Lieq, \Lien)=0$ guarantees the well-definition of $\sigma$. Moreover $\sigma$ is a surjection, which completes the proof.
\end{proof}

\begin{Co} \label{several properties}
Let $\Lien$ be a two-sided ideal of a finite-dimensional Leibniz algebra $\Lieq$ together with the free presentations in diagram  (\ref{free present diagr}). Then
\begin{enumerate}
\item[(a)] ${\cal M}_{\Lie}^{(c)} (\Lieq)$ is finite-dimensional.

\item[(b)] ${\rm dim} \left(  {\cal M}_{\Lie}^{(c)} \left( \frac{\Lieq}{\Lien} \right) \right) \leq {\rm dim} \left( {\cal M}_{\Lie}^{(c)} (\Lieq) \right) +
{\rm dim} \left( \frac{ \Lien \cap  \gamma_{c+1}^{\Lie}(\Lieq)}{\gamma_{c+1}^{\Lie}(\Lieq, \Lien)} \right).$

\item[(c)] ${\rm dim} \left( {\cal M}_{\Lie}^{(c)} (\Lieq) \right) + {\rm dim} \left( \Lien \cap \gamma_{c+1}^{\Lie}(\Lieq) \right) = {\rm dim} \left(  {\cal M}_{\Lie}^{(c)} \left( \frac{\Lieq}{\Lien} \right) \right)  + {\rm dim} \left( \gamma_{c+1}^{\Lie}(\Lieq, \Lien) \right) + {\rm dim} \left( \frac{ \Lier \cap  \gamma_{c+1}^{\Lie}(\Lief, \Lies)}{\gamma_{c+1}^{\Lie}(\Lief, \Lier)} \right)$.

\item[(d)] ${\rm dim} \left( {\cal M}_{\Lie}^{(c)} (\Lieq) \right) + {\rm dim} \left( \Lien \cap \gamma_{c+1}^{\Lie}(\Lieq) \right) = {\rm dim} \left(  {\cal M}_{\Lie}^{(c)} \left( \frac{\Lieq}{\Lien} \right) \right)  + {\rm dim} \left( \frac{ \gamma_{c+1}^{\Lie}(\Lief, \Lies)}{\gamma_{c+1}^{\Lie}(\Lief, \Lier)} \right)$.

    \item[(e)] ${\rm dim} \left( {\cal M}_{\Lie}^{(c)} (\Lieq) \right) + {\rm dim} \left( \gamma_{c+1}^{\Lie}(\Lieq) \right) = {\rm dim} \left( \gamma_{c+1}^{\Lie} \left(\frac{ \Lief}{\gamma_{c+1}^{\Lie}(\Lief, \Lier)} \right) \right).$

        \item[(f)] If ${\cal M}_{\Lie}^{(c)} (\Lieq) =0$, then ${\cal M}_{\Lie}^{(c)} \left( \frac{\Lieq}{\Lien} \right)  \cong \frac{ \Lien \cap  \gamma_{c+1}^{\Lie}(\Lieq)}{\gamma_{c+1}^{\Lie}(\Lieq, \Lien)}.$

    \item[(g)]  ${\rm dim} \left( {\cal M}_{\Lie}^{(c)} (\Lieq) \right) + {\rm dim} \left( \Lien \cap \gamma_{c+1}^{\Lie}(\Lieq) \right)  \leq {\rm dim} \left(  {\cal M}_{\Lie}^{(c)} \left( \frac{\Lieq}{\Lien} \right) \right)  + {\rm dim} \left( \Lien \otimes^c \Lieq_{\Lie}  \right)$,  provided that $0 \to \Lien \to \Lieq \to \frac{\Lieq}{\Lien} \to 0$ is a $c$-{\Lie}-central extension.
\end{enumerate}
\end{Co}
\begin{proof}
{\it (a)}  Straightforward since $\Lieq$ is a finite-dimensional Leibniz algebra.

\noindent {\it (b)} From Proposition \ref{exact sequences} {\it (a)}, the exact sequence yields
 \begin{equation} \label{eq}
{\rm dim} \left(  {\cal M}_{\Lie}^{(c)} (\frac{\Lieq}{\Lien}) \right) + {\rm dim} \left( \frac{ \Lier \cap  \gamma_{c+1}^{\Lie}(\Lief,\Lies)}{\gamma_{c+1}^{\Lie}(\Lief, \Lier)} \right)= {\rm dim} \left( {\cal M}_{\Lie}^{(c)} (\Lieq) \right) +
{\rm dim} \left( \frac{ \Lien \cap  \gamma_{c+1}^{\Lie}(\Lieq)}{\gamma_{c+1}^{\Lie}(\Lieq, \Lien)} \right)
\end{equation}
 The result follows.

\noindent {\it (c)}  The result follows from  the equation (\ref{eq}) above combined with the equality
${\rm dim} \left( \frac{ \Lien \cap  \gamma_{c+1}^{\Lie}(\Lieq)}{\gamma_{c+1}^{\Lie}(\Lieq, \Lien)} \right)={\rm dim} \left( \Lien \cap \gamma_{c+1}^{\Lie}(\Lieq) \right) -  {\rm dim} \left( \gamma_{c+1}^{\Lie}(\Lieq, \Lien) \right).$

\noindent {\it (d)} The result holds from {\it (c)}, since from diagram (\ref{diagram1}) and by the isomorphism theorem we have
$$\gamma_{c+1}^{\Lie}(\Lieq, \Lien) \cong \frac{\gamma_{c+1}^{\Lie}(\Lief, \Lies)}{\Lier \cap \gamma_{c+1}^{\Lie}(\Lief, \Lies)} \cong
\frac{\gamma_{c+1}^{\Lie}(\Lief, \Lies)/\gamma_{c+1}^{\Lie}(\Lief, \Lier)}{\Lier \cap \gamma_{c+1}^{\Lie}(\Lief, \Lies)/\gamma_{c+1}^{\Lie}(\Lief, \Lier)}$$

\noindent{\it (e)} Apply statement {\it (d)} to the particular case $\Lien = \Lieq$.

\noindent{\it  (f)} Straightforward from  Proposition \ref{exact sequences} {\it (a)}.

\noindent{\it  (g)} Straightforward from exact sequence in Proposition 4.1 {\it (c)}.
\end{proof}

\begin{De} \label{filiform}
 A $\Lie$-nilpotent Leibniz algebra $\Lieq$ of class $c$ is said to be of maximal $\Lie$-class $c$ if ${\rm dim} \left( \frac{\gamma_j^{\Lie}(\Lieq)}{\gamma_{j+1}^{\Lie}(\Lieq)} \right) = 1$, for $j= 2, 3, \dots, c$ and ${\rm dim} \left( \frac{\Lieq}{\gamma_{2}^{\Lie}(\Lieq)} \right) = 2$.
\end{De}

\begin{Rem} The absolute case of Definition \ref{filiform}, that is when the Liezation functor is substituted by the abelianization functor, is equivalent to the notion of filiform Leibniz algebra \cite{AO}.
\end{Rem}

\begin{Pro}
Let $\Lieq$ be a $\Lie$-nilpotent Leibniz algebra of class $c$.
If $\Lieq$ is of maximal $\Lie$-class $c$, then ${\ze}_i^{\Lie}({\Lieq}) = \gamma_{c-i+1}^{\Lie}(\Lieq)$, for $0 \leq i \leq c$.
\end{Pro}
\begin{proof}
The statement is true for $i\in \{0,\,  c\}$ since ${\ze}_0^{\Lie}({\Lieq}) = 0=\gamma_{c+1}^{\Lie}(\Lieq),$ and ${\ze}_c^{\Lie}({\Lieq}) = \Lieq=\gamma_{1}^{\Lie}(\Lieq)$.

 By induction, assume that ${\ze}_{i-1}^{\Lie}({\Lieq}) = \gamma_{c-i+2}^{\Lie}(\Lieq)$, for $1 \leq i < c$. Then ${\ze}_i^{\Lie}({\Lieq}) =C_{\Lieq}^{\Lie}({\Lieq},{\ze}_{i-1}^{\Lie}({\Lieq}))=C_{\Lieq}^{\Lie}({\Lieq},\gamma_{c-i+2}^{\Lie}(\Lieq))$. So $\gamma_{c-i+1}^{\Lie}(\Lieq)\subseteq {\ze}_i^{\Lie}({\Lieq}).$
 Now, it is easy to check that $\frac{{\ze}_i^{\Lie}({\Lieq})}{\gamma_{c-i+1}^{\Lie}(\Lieq)} \varsubsetneq \frac{\gamma_{c-i}^{\Lie}(\Lieq)}{\gamma_{c-i+1}^{\Lie}(\Lieq)}$. Indeed,  since ${\rm dim} \left(\frac{\gamma_{c-i+1}^{\Lie}(\Lieq)}{\gamma_{c-i+2}^{\Lie}(\Lieq)} \right) \neq 0,$ it follows that $\gamma_{c-i+2}^{\Lie}(\Lieq) \varsubsetneq \gamma_{c-i+1}^{\Lie}(\Lieq)$. So let $x \in \gamma_{c-i+1}^{\Lie}(\Lieq)\backslash \gamma_{c-i+2}^{\Lie}(\Lieq)$, then $x = [x_0, m_0]_{\Lie}$, for some $x_0 \in \gamma_{c-i}^{\Lie}(\Lieq)$ and $m_0 \in \Lieq$. Then $x_0 \notin {\ze}_i^{\Lie}(\Lieq)$, otherwise $x =  [x_0, m_0]_{\Lie} \in {\ze}_{i-1}^{\Lie}(\Lieq) = \gamma_{c-i+2}^{\Lie}(\Lieq)$ which is a contradiction.
 Since ${\rm dim} \left(\frac{\gamma_{c-i}^{\Lie}(\Lieq)}{\gamma_{c-i+1}^{\Lie}(\Lieq)} \right) =1,$ it follows that $\frac{{\ze}_i^{\Lie}({\Lieq})}{\gamma_{c-i+1}^{\Lie}(\Lieq)}=0$, hence ${\ze}_i^{\Lie}({\Lieq}) = \gamma_{c-i+1}^{\Lie}(\Lieq)$.
\end{proof}

\begin{Co}
Let $\Lieq$ be a finite dimensional $\Lie$-nilpotent Leibniz algebra of maximal $\Lie$-class $c+1$, then
$${\rm dim} \left( {\cal M}_{\Lie}^{(c)} (\Lieq) \right)   \leq {\rm dim} \left(  {\cal M}_{\Lie}^{(c)} \left(\frac{\Lieq}{Z_{\Lie}(\Lieq)} \right) \right)  + 2^c - 1$$
\end{Co}
\begin{proof}
 Letting $\Lien:=Z_{\Lie}(\Lieq)$ in Corollary \ref{several properties} {\it (g)}, we have $${\rm dim} \left( {\cal M}_{\Lie}^{(c)} (\Lieq) \right)+{\rm dim} \left( Z_{\Lie}(\Lieq) \cap \gamma_{c+1}^{\Lie}(\Lieq) \right)\leq {\rm dim} \left({\cal M}_{\Lie}^{(c)} \left( \frac{\Lieq}{Z_{\Lie}(\Lieq)} \right) \right)+{\rm dim} \left( \Lien \otimes^c \Lieq_{\Lie}\right).$$
Note that since $\Lieq$ is $\Lie$-nilpotent  of maximal  $\Lie$-class $c+1,$ it follows that $Z_{\Lie}(\Lieq)={\ze}_1^{\Lie}({\Lieq}) =\gamma_{c+1}^{\Lie}(\Lieq),$  $\gamma_{c+2}^{\Lie}(\Lieq)=0$, then ${\rm dim} \left( Z_{\Lie}(\Lieq) \right) = {\rm dim} \left( \frac{ \gamma_{c+1}^{\Lie}(\Lieq)}{\gamma_{c+2}^{\Lie}(\Lieq)} \right) = 1$ and
 ${\rm dim} \left( \Lieq_{\Lie} \right)={\rm dim} \left( \frac{\Lieq}{\gamma_{2}^{\Lie}(\Lieq)} \right) = 2.$
 Therefore ${\rm dim} \left( Z_{\Lie}(\Lieq) \cap \gamma_{c+1}^{\Lie}(\Lieq) \right)={\rm dim} \left( Z_{\Lie}(\Lieq) \right) = 1$ and ${\rm dim} \left(  Z_{\Lie}(\Lieq) \otimes^c \Lieq_{\Lie}\right) = 2^c$. Then the result follows.
\end{proof}

\section{$c$-$\Lie$-stem covers} \label{stem}

 In this section we analyze the interplay between $c$-$\Lie$-stem covers and the $c$-nilpotent Schur $\Lie$-multiplier.

\begin{De}
A $c$-$\Lie$-central extension  $0 \to \Lien \to \Lieg \stackrel{\pi} \to \Lieq \to 0$ is said to be $c$-$\Lie$-stem  extension whenever $\Lien \subseteq \gamma_{c+1}^{\Lie}(\Lieg)$.

In addition, if $\Lien$ is isomorphic to ${\cal M}_{\Lie}^{(c)} (\Lieq)$, then the  $c$-$\Lie$-stem  extension  is called a $c$-$\Lie$-stem  cover of $\Lieq$. In this case $\Lieg$ is said to be a $c$-$\Lie$-covering of $\Lieq$.

A Leibniz algebra $\Lieq$ is said to be Hopfian is every surjective homomorphism $\Lieq \twoheadrightarrow \Lieq$ is an isomorphism.
\end{De}

\begin{Pro} \label{stem ext}
For a $c$-$\Lie$-central extension $\pi : \Lieg \twoheadrightarrow \Lieq$, with $\Lien = {\ker}(\pi)$, the following statements are equivalent:
\begin{enumerate}
\item[(a)] $\pi : \Lieg \twoheadrightarrow \Lieq$ is a $c$-$\Lie$-stem extension.
\item[(b)] The induced map $\Lien \to \frac{\Lieg}{\gamma_{c+1}^{\Lie}(\Lieg)}$ is the zero map.
\item[(c)] $\theta : {\cal M}_{\Lie}^{(c)} (\Lieq) \to \Lien$ is an epimorphism.
\item[(d)] The following sequence $\n \otimes^c \Lieg_{\Lie} \to {\cal M}_{\Lie}^{(c)} (\Lieg) \to {\cal M}_{\Lie}^{(c)} (\Lieq) \stackrel{\theta} \to \frak{n} \to 0$ is exact.
\item[(e)]$\frac{\Lieg}{\gamma_{c+1}^{\Lie}(\Lieg)} \cong \frac{\Lieq}{\gamma_{c+1}^{\Lie}(\Lieq)}$.
\end{enumerate}
\end{Pro}
\begin{proof} The equivalences between {\it (a), (b), (c)} and {\it (d)} follow from exact sequences in Proposition \ref{exact sequences}. The equivalence between {\it (a)} and {\it (e)} is a consequence of the following $3 \times 3$ diagram:
\[ \xymatrix{
 & 0 \ar[d] & 0 \ar[d] & 0 \ar[d] & \\
0 \ar[r] & \n \cap \gamma_{c+1}^{\Lie}(\Lieg) \ar[r] \ar[d] & \gamma_{c+1}^{\Lie}(\Lieg) \ar[r] \ar[d] & \gamma_{c+1}^{\Lie}(\Lieq) \ar[r] \ar[d] & 0\\
0 \ar[r] & \n \ar[r] \ar[d] & \g \ar[r] \ar[d] & \q \ar[r] \ar[d] & 0\\
0 \ar[r] & 0 \ar[r] \ar[d] & \frac{\Lieg}{\gamma_{c+1}^{\Lie}(\Lieg)} \ar@{=}[r] \ar[d] &\frac{\Lieq}{\gamma_{c+1}^{\Lie}(\Lieq)} \ar[r] \ar[d] & 0\\
 & 0  & 0 & 0  &
} \]
\end{proof}

\begin{Pro} \label{stem cover}
For a  $c$-$\Lie$-central extension $\pi : \Lieg \twoheadrightarrow \Lieq$  the following statements are equivalent:
\begin{enumerate}
\item[(a)] $\pi : \Lieg \twoheadrightarrow \Lieq$  is a $c$-$\Lie$-stem cover.
\item[(b)] $\frac{\Lieg}{\gamma_{c+1}^{\Lie}(\Lieg)} \cong \frac{\Lieq}{\gamma_{c+1}^{\Lie}(\Lieq)}$ and the induced map ${\cal M}_{\Lie}^{(c)} (\Lieg) \to {\cal M}_{\Lie}^{(c)} (\Lieq)$  is the zero map.
\end{enumerate}
\end{Pro}
\begin{proof} This is a direct consequence of Proposition \ref{stem ext} {\it (e)} and Proposition \ref{exact sequences}.
\end{proof}

\begin{Ex}\ \label{Example3.14}
\begin{enumerate}
\item[(a)]  The $2$-$\Lie$-central extension given in Example \ref{example Lie stem} {\it (b)} is  a $2$-$\Lie$-stem extension, but the $2$-$\Lie$-central  in Example \ref{example Lie stem} {\it (a)} is not a $c$-$\Lie$-stem extension.

\item[(b)] Let $0 \to \Lier \to \Lief \stackrel{\rho} \to \Lieq \to 0$ be a free presentation of a Leibniz algebra $\Lieq$. Then $\gamma_{c+1}^{\Lie}(\Lief, \Lier)$ is a two-sided ideal of $\Lief$, $\gamma_{c+1}^{\Lie}(\Lief, \Lier) \subseteq \Lier$ and the sequence
    \begin{equation}\label{example-Lie-central}
 0 \longrightarrow \frac{\Lier}{\gamma_{c+1}^{\Lie}(\Lief, \Lier)} \longrightarrow \frac{\Lief}{\gamma_{c+1}^{\Lie}(\Lief, \Lier)} \stackrel{\overline{\rho}} \longrightarrow \Lieq \longrightarrow 0
    \end{equation}
      is a $c$-$\Lie$-central extension (see Example \ref{example exact sequence} {\it (b)}).

 It has the property that the induced map ${\cal M}_{\Lie}^{(c)}\left(\Lief/\gamma_{c+1}^{\Lie}(\Lief, \Lier) \right) \to  {\cal M}_{\Lie}^{(c)}(\Lieq)$ is the zero map. This can be readily checked by using the isomorphism (\ref{Schur formula}) for the given free presentation of $\Lieq$ and the free presentation $0 \to \gamma_{c+1}^{\Lie}(\Lief, \Lier) \to \Lief  \to \Lief/\gamma_{c+1}^{\Lie}(\Lief, \Lier) \to 0$ of the Leibniz algebra $\Lief/\gamma_{c+1}^{\Lie}(\Lief, \Lier)$. Moreover, we have the short exact sequence (c.f. the last row of $3\times 3$ diagram in the proof of Proposition \ref{stem ext})
     \begin{equation}\label{kernel-h1}
    0\longrightarrow \frac{\Lier}{\Lier\cap \gamma_{c+1}^{\Lie}(\Lief)}\longrightarrow \frac{\Lief}{\gamma_{c+1}^{\Lie}(\Lief)} \longrightarrow \frac{\Lieq}{\gamma_{c+1}^{\Lie}(\Lieq)}\longrightarrow 0.
     \end{equation}

\item[(c)] As a particular case of {\it (b)}, consider the Leibniz  algebra $\Lieq=\Lief/\gamma_{c+1}^{\Lie}(\Lief)$ for a free Leibniz algebra $\Lief$. Then $\Lier=\gamma_{c+1}^{\Lie}(\Lief)$ and the sequence
(\ref{example-Lie-central}) turns to
\[
0 \longrightarrow \frac{\gamma_{c+1}^{\Lie}(\Lief)}{\gamma_{c+1}^{\Lie}(\Lief, \Lier)} \longrightarrow \frac{\Lief}{\gamma_{c+1}^{\Lie}(\Lief, \Lier)} \longrightarrow \frac{\Lief}{\gamma_{c+1}^{\Lie}(\Lief)} \longrightarrow 0,
\]
which, by (\ref{kernel-h1}), is a $c$-$\Lie$-stem cover of the Leibniz algebra  $\Lief/\gamma_{c+1}^{\Lie}(\Lief)$.

\end{enumerate}
\end{Ex}

\begin{Le} \label{equivalence property}
Let $0 \to \Lier \to \Lief \overset{\rho} \to \Lieq \to 0$ be a free presentation of a Leibniz algebra $\Lieq$. Then the extension $0 \to  \Liem \to \Lieq^{\ast} \overset{\psi} \to \Lieq \to 0$ is a $c$-\Lie-stem cover  of $\Lieq$ if and only if there exists a two-sided ideal $\Lies$ of $\Lief$ such that
\begin{enumerate}
\item[(a)] $\Lieq^{\ast} \cong \Lief/\Lies$ and $\Liem \cong \Lier/\Lies$;

\item[(b)] $\Lier/\gamma_{c+1}^{\Lie}(\Lief, \Lier) \cong {\cal M}_{\Lie}^{(c)} (\Lieq) \oplus \Lies/\gamma_{c+1}^{\Lie}(\Lief, \Lier)$.
\end{enumerate}
\end{Le}
\begin{proof}
Let $0 \to  \Liem \to \Lieq^{\ast} \overset{\psi} \to \Lieq \to 0$ be a  $c$-\Lie-stem cover  of $\Lieq.$ Then by Lemma 3.6, the identity map $\Lieq\to\Lieq$ induces  a homomorphism  $\beta : \frac{\Lief}{\gamma_{c+1}^{\Lie}(\Lief, \Lier)} \to \Lieq^{\ast}$ such that $\beta \left( \frac{\Lier}{\gamma_{c+1}^{\Lie}(\Lief, \Lier)} \right) \subseteq \Liem$ and $\psi \circ \beta=\bar{\rho}$ (see diagram (\ref{diagram lemma})). Since $\Lieq^{\ast} = \im(\beta) + \Liem$ and $\Liem \subseteq \gamma_{c+1}^{\Lie}(\Lieq^{\ast}) = \gamma_{c+1}^{\Lie}(\im(\beta)) \subseteq \im(\beta)$, hence $\beta$ is a surjective homomorphism and $\beta \left( \frac{\Lier}{\gamma_{c+1}^{\Lie}(\Lief, \Lier)} \right) = \Liem$.

Now, let $\Lies$ be a two-sided ideal of $\Lief$ such that $\ker(\beta)= \frac{\Lies}{\gamma_{c+1}^{\Lie}(\Lief, \Lier)}.$ Then we have the  exact sequence $0\to \frac{\Lies}{\gamma_{c+1}^{\Lie}(\Lief, \Lier)}\to  \frac{\Lief}{\gamma_{c+1}^{\Lie}(\Lief, \Lier)}\overset{\beta}\to \Lieq^{\ast}  \to 0$ which induces the short exact sequence $ 0\to \frac{\Lies}{\gamma_{c+1}^{\Lie}(\Lief, \Lier)} \to \frac{\Lier}{\gamma_{c+1}^{\Lie}(\Lief, \Lier)}\overset{\beta_{\mid \ }} \to \Liem \to 0.$ It follows from these two exact sequences and the third isomorphism theorem that $\Lieq^{\ast}\cong\frac{\Lief}{\gamma_{c+1}^{\Lie}(\Lief, \Lier)}/\frac{\Lies}{\gamma_{c+1}^{\Lie}(\Lief, \Lier)}\cong \Lief/\Lies$ and  $\Liem\cong\frac{\Lier}{\gamma_{c+1}^{\Lie}(\Lief, \Lier)}/\frac{\Lies}{\gamma_{c+1}^{\Lie}(\Lief, \Lier)}\cong\Lier/\Lies.$ Moreover,  $\frac{\Lier}{\gamma_{c+1}^{\Lie}(\Lief, \Lier)}\cong \Liem\oplus\frac{\Lies}{\gamma_{c+1}^{\Lie}(\Lief, \Lier)}$ as $\mathbb{K}$-vector spaces, and thus $\frac{\Lier}{\gamma_{c+1}^{\Lie}(\Lief, \Lier)}\cong  {\cal M}_{\Lie}^{(c)} (\Lieq)\oplus\frac{\Lies}{\gamma_{c+1}^{\Lie}(\Lief, \Lier)}, $ since $0 \to  \Liem \to \Lieq^{\ast} \overset{\psi} \to \Lieq \to 0$ is a  $c$-$\Lie$-stem cover  of $\Lieq.$

Conversely, suppose the existence of a two-sided ideal $\Lies$ of $\Lief$ satisfying {\it (a)} and {\it (b)}. Then, $\frac{\Lieq^{\ast}}{\Liem} \cong \frac{\Lief}{\Lies}/\frac{\Lier}{\Lies}\cong \Lief/\Lier\cong\Lieq,$ and   ${\cal M}_{\Lie}^{(c)} (\Lieq)\cong\frac{\Lier}{\gamma_{c+1}^{\Lie}(\Lief, \Lier)}/\frac{\Lies}{\gamma_{c+1}^{\Lie}(\Lief, \Lier)}\cong\Lier/\Lies\cong\Liem.$ Moreover $\Liem\cong\Lier/\Lies\subseteq\frac{\Lies+\gamma_{c+1}^{\Lie}(\Lief)}{\Lies}\cong\frac{\gamma_{c+1}^{\Lie}(\Lief)}{\Lies\cap\gamma_{c+1}^{\Lie}(\Lief)}\subseteq \gamma_{c+1}^{\Lie}(\Lief/ \Lies)  \cong\gamma_{c+1}^{\Lie}(\Lieq^{\ast}). $ Therefore the extension $0 \to  \Liem \to \Lieq^{\ast} \overset{\psi} \to \Lieq \to 0$ is a $c$-\Lie-stem cover  of $\Lieq.$
\end{proof}

\begin{Co} \label{existence}
Any finite-dimensional Leibniz algebra has at least one $c$-$\Lie$-covering.
\end{Co}
\begin{proof}
Let $\Lieq$ be a finite dimensional Leibniz algebra, and let  $0 \to \Lier \to \Lief \overset{\rho} \to \Lieq \to 0$ be a free presentation of $\Lieq$. Following the proof of Lemma \ref{equivalence property}, choose a two-sided ideal $\Lies$ of $\Lief$ such that $\frac{\Lies}{\gamma_{c+1}^{\Lie}(\Lief, \Lier)}$ is the complement of $ {\cal M}_{\Lie}^{(c)} (\Lieq) $ in $\frac{\Lier}{\gamma_{c+1}^{\Lie}(\Lief, \Lier)}.$ Then the extension $0\to\Lier/ \Lies\to\Lief/ \Lies\to\Lieq\to0$ is a  $c$-\Lie-stem cover  of $\Lieq.$
\end{proof}

\begin{Th}
Let $0 \to  \Liem \to \Lieh \overset{\theta} \to \Lieq \to 0$ be a $c$-$\Lie$-stem extension of a finite-dimensional Leibniz algebra $\Lieq$. Then there exists a $c$-$\Lie$-covering $\Lieq^{\ast}$ of $\Lieq$ such that $\Lieh$ is a quotient of $\Lieq^{\ast}$.
\end{Th}
\begin{proof}
Let  $0 \to \Lier \to \Lief \overset{\rho} \to \Lieq \to 0$ be a free presentation of $\Lieq$. Using Lemma \ref{lema 1}, choose a two-sided ideal $\Lies_0$ of $\Lief$ such that $\ker(\beta)=\frac{\Lies_0}{\gamma_{c+1}^{\Lie}(\Lief, \Lier)}$. Since the surjective homomorphism $\beta$ is induced by the identity map, it follows that $\ker(\beta)=\ker(\beta_{\mid})$, hence $\Lies_0$ in fact is a two-sided ideal of $\Lier$. Following the proof of Lemma \ref{equivalence property}, we therefore have that $\frac{\Lies_0}{\gamma_{c+1}^{\Lie}(\Lief, \Lier)}$ is the complement of $ {\cal M}_{\Lie}^{(c)} (\Lieq) $ in $\frac{\Lier}{\gamma_{c+1}^{\Lie}(\Lief, \Lier)}.$
So $(\Lier\cap\gamma_{c+1}^{\Lie}(\Lief))+\Lies_0=\Lier$ and $(\Lier\cap\gamma_{c+1}^{\Lie}(\Lief))\cap\Lies_0=\gamma_{c+1}^{\Lie}(\Lief, \Lier).$ Now let $\Lies$ be a two-sided ideal of $\Lies_0$ such that $\frac{\Lies}{\gamma_{c+1}^{\Lie}(\Lief, \Lier)}$ is the complement of $\frac{\Lies_0 \cap {\gamma_{c+1}^{\Lie}(\Lief)}}{{\gamma_{c+1}^{\Lie}(\Lief,\Lier)}}$ in $\frac{\Lies_0}{\gamma_{c+1}^{\Lie}(\Lief, \Lier)}.$ So $(\Lies_0\cap\gamma_{c+1}^{\Lie}(\Lief))+\Lies=\Lies_0$ and $(\Lies_0\cap\gamma_{c+1}^{\Lie}(\Lief))\cap\Lies=\gamma_{c+1}^{\Lie}(\Lief, \Lier).$ This implies that  $(\Lier\cap\gamma_{c+1}^{\Lie}(\Lief))+\Lies=\Lier$ and $(\Lier\cap\gamma_{c+1}^{\Lie}(\Lief))\cap\Lies=\gamma_{c+1}^{\Lie}(\Lief, \Lier)$ since $\Lies\subseteq\Lies_0\subseteq\Lier.$ Therefore $\Lier/\gamma_{c+1}^{\Lie}(\Lief, \Lier) \cong {\cal M}_{\Lie}^{(c)} (\Lieq) \oplus \Lies/\gamma_{c+1}^{\Lie}(\Lief, \Lier)$. Hence $\Lieq^{\ast}:=\Lief/\Lies$ is a $c$-$\Lie$-stem cover  of $\Lieq$ by Lemma \ref{equivalence property}. Clearly $$\frac{\Lieq^{\ast}}{\Lies_0/\Lies}\cong\Lief/\Lies_0\cong\frac{\Lief}{\gamma_{c+1}^{\Lie}(\Lief, \Lier)}/\frac{\Lies_0}{\gamma_{c+1}^{\Lie}(\Lief, \Lier)}=\frac{\Lief}{\gamma_{c+1}^{\Lie}(\Lief, \Lier)}/\ker(\beta)\cong\Lieh.$$ Hence $\Lieh$ is a quotient of $\Lieq^{\ast}$.
\end{proof}

\begin{Le} \label{commutative diagram}
Let $\Lieq$ be a Leibniz algebra and
\[ \xymatrix{
0 \ar[r] & \Lien_1 \ar[r] \ar[d]_{\alpha} & \Lieg_1 \ar[r] \ar[d]_{\beta} & \Lieq \ar[r] \ar[d]_{\gamma} & 0\\
0 \ar[r] & \Lien_2 \ar[r]  & \Lieg_2 \ar[r]  & \Lieq \ar[r] & 0
} \]
be a commutative diagram of short exact sequences of Leibniz algebras such that the bottom row is a $c$-$\Lie$-stem extension. If the homomorphism $\gamma$ is surjective, then $\beta$ is a surjective homomorphism as well.
\end{Le}
\begin{proof}
Obviously $\Lieg_2 = \im(\beta) + \Lien_2$.
Since $\Lien_2 \subseteq \gamma_{c+1}^{\Lie}(\Lieg_2)$ and $\gamma_{c+1}^{\Lie}(\Lieg_2, \Lien_2)=0$, then   $\Lien_2 \subseteq \beta \left( \gamma_{c+1}^{\Lie}(\Lieg_1) \right)$.
Therefore  $\Lieg_2 \subseteq \im(\beta) +  \beta \left( \gamma_{c+1}^{\Lie}(\Lieg_1) \right)$, i.e. $\beta$ is surjective.
\end{proof}

\begin{Th} \label{epimorphism1}
Let $\Lieq$ be a Leibniz algebra and let $0 \to  \Liem_i \to \Lieh_i \overset{\theta_i} \to \Lieq \to 0, i =1, 2$, be two $c$-$\Lie$-stem covers of $\Lieq$. If $\eta : \Lieh_1 \to \Lieh_2$ is a surjective homomorphism such that $\eta(\Liem_1) \subseteq \Liem_2$, then $\eta$ is an isomorphism.
\end{Th}
\begin{proof}
 Let $0 \to \Lier \to \Lief \overset{\rho} \to \Lieq \to 0$ be a free presentation of $\Lieq$. Then by Lemma \ref{equivalence property}, there exist two-sided ideals $\Lies_i$ of $\Lief,$ $i=1,2,$ such that  $\Lieh_i\cong \Lief/\Lies_i,$  $\Liem_i \cong \Lier/\Lies_i$ and  $\Lier/\gamma_{c+1}^{\Lie}(\Lief, \Lier) \cong {\cal M}_{\Lie}^{(c)} (\Lieq) \oplus \Lies_i/\gamma_{c+1}^{\Lie}(\Lief, \Lier).$ It is therefore enough to prove that the surjective homomorphism  $\bar{\eta}:\Lief/\Lies_1\to\Lief/\Lies_2$ induced by $\eta$ is an isomorphism. Following the proof of Lemma \ref{equivalence property}, we have a surjective homomorphism  $\beta_2 : \frac{\Lief}{\gamma_{c+1}^{\Lie}(\Lief, \Lier)} \to \Lieh_2$ such that $\beta_2 \left( \frac{\Lier}{\gamma_{c+1}^{\Lie}(\Lief, \Lier)} \right) = \Liem_2$ and $\ker (\beta_2)= \frac{\Lies_2}{\gamma_{c+1}^{\Lie}(\Lief, \Lier)}$. Since $\Lief$ is a free Leibniz algebra, there exists a homomorphism $\tilde{\delta}:\Lief\to\Lief/\Lies_1$ such that the following  diagram is commutative,
  \begin{equation}
  \xymatrix{
     \\
 &  & \frac{\Lief}{\gamma_{c+1}^{\Lie}(\Lief, \Lier)} \ar@{>>}[dd]^{\delta} \ar@{>>}[rd]^{\beta_2} \\
& {\Lief} \ar@{>>}[ru]^{nat} \ar[rd]^{\tilde{\delta}}& & \Lief/\Lies_2   &  \\
   &  & \Lief/\Lies_1 \ar@{>>}[ru]^{\bar{\eta}}&
  }
  \end{equation}where $\delta$ is a homomorphism induced by $\tilde{\delta}.$ Since $\theta_1 \circ \delta = \overline{\rho}$, then $\delta$ is a surjective homomorphism by Lemma \ref{commutative diagram}

 In addition, if $\Lies$ is a two-sided ideal of $\Lief$ such that  $\ker(\delta)= \frac{\Lies}{\gamma_{c+1}^{\Lie}(\Lief, \Lier)},$ then $\frac{\Lies}{\gamma_{c+1}^{\Lie}(\Lief, \Lier)}$ is the complement of $ {\cal M}_{\Lie}^{(c)} (\Lieq) $ in $\frac{\Lier}{\gamma_{c+1}^{\Lie}(\Lief, \Lier)},$ which implies that $(\Lier\cap\gamma_{c+1}^{\Lie}(\Lief))+\Lies=\Lier.$ Now, one easily shows that $\Lies\subseteq\Lies_2.$ Therefore $\Lies=\Lies_2$ by Lemma \ref{equivalence property}. Hence $\ker(\delta)=\ker(\beta_2),$ which implies that $\bar{\eta}$ is injective.
\end{proof}

\begin{Co} \label{hopfian}
Let $\Lieq$ be a nilpotent Leibniz algebra of class $c \geq 1$, then every $c$-$\Lie$-covering of $\Lieq$ is Hopfian.
\end{Co}
\begin{proof}
Let $0 \to  \Liem \to \Lieq^{\ast} \overset{\psi} \to \Lieq \to 0$ is a $c$-\Lie-stem cover  of $\Lieq.$  Then by Proposition \ref{stem cover},  $\frac{\Lieq^{\ast}}{\gamma_{c+1}^{\Lie}(\Lieq^{\ast})} \cong \frac{\Lieq}{\gamma_{c+1}^{\Lie}(\Lieq)}\cong \Lieq$ since  $\Lieq$ is  a nilpotent Leibniz algebra of class $c$,  hence $0 \to \gamma_{c+1}^{\Lie}(\Lieq^{\ast}) \to \Lieq^{\ast} \to \Lieq \to 0$ is a $c$-$\Lie$-stem cover. Now let $\eta: \Lieq^{\ast} \twoheadrightarrow \Lieq^{\ast}$ be a surjective homomorphism, then $\eta(\Liem) \subseteq \gamma_{c+1}^{\Lie}(\Lieq^{\ast})$ since $\Liem \subseteq \gamma_{c+1}^{\Lie}(\Lieq^{\ast})$. So $\eta([\Lieq^{\ast},\Lieq^{\ast}]_{\Lie})=[\eta(\Lieq^{\ast}),\eta(\Lieq^{\ast})]_{\Lie}=[\Lieq^{\ast},\Lieq^{\ast}]_{\Lie},$ and thus $\eta(\gamma_{c+1}^{\Lie}(\Lieq^{\ast}))=\gamma_{c+1}^{\Lie}(\Lieq^{\ast}).$  It follows now by Theorem \ref{epimorphism1} that $\eta$ is an isomorphism, hence $\Lieq^{\ast}$ is Hophian.
\end{proof}

\begin{Th}
Let $0 \to  \Liem_i \to \Lieh_i \overset{\theta_i} \to \Lieq \to 0, i =1, 2$, be two $c$-$\Lie$-stem covers of a finite-dimensional Leibniz algebra $\Lieq$. Then ${\cal Z}_{c+1}^{\Lie}(\Lieh_1)/\Liem_1 \cong {\cal Z}_{c+1}^{\Lie}(\Lieh_2)/\Liem_2$.
\end{Th}

\begin{proof}
Let  $0 \to \Lier \to \Lief \overset{\rho} \to \Lieq \to 0$ be a free presentation of $\Lieq,$ and $\Lieq^{\ast}$ a $c$-$\Lie$-covering of $\Lieq$, which exists by Corollary \ref{existence}. Then by Lemma \ref{equivalence property}, there exists a two-sided ideal $\Lies$ of $\Lief$ such that   $\Lieq^{\ast} \cong \Lief/\Lies$ and $\Liem \cong \Lier/\Lies$, and
$\Lier/\gamma_{c+1}^{\Lie}(\Lief, \Lier) \cong {\cal M}_{\Lie}^{(c)} (\Lieq) \oplus \Lies/\gamma_{c+1}^{\Lie}(\Lief, \Lier)$. Now let $\Liet$ be a two-sided ideal of $\Lief$ such that  ${\cal Z}_{c+1}^{\Lie}( {\Lief}/{\gamma_{c+1}^{\Lie}(\Lief, \Lier)})= {\Liet}/{\gamma_{c+1}^{\Lie}(\Lief, \Lier)}.$  This implies that   for all $t\in\Liet$ and  $f_i\in\Lief, 1\leq i\leq c,$  we have  $[[[t+\Lies,f_1+\Lies]_{\Lie},f_2+\Lies]_{\Lie}, \dots, f_c+\Lies]_{\Lie}=[[[t,f_1]_{\Lie},f_2]_{\Lie}, \dots, f_c]_{\Lie}+\Lies\in\gamma_{c+1}^{\Lie}(\Lief, \Lier)+\Lies=\Lies$, and thus $\Liet/\Lies\subseteq{\cal Z}_{c+1}^{\Lie}( {\Lief}/{\Lies}).$  We claim that ${\cal Z}_{c+1}^{\Lie}( {\Lief}/{\Lies})\subseteq\Liet/ \Lies.$  Indeed, let $x+\Lies\in {\cal Z}_{c+1}^{\Lie}( {\Lief}/{\Lies}).$ Then  for all $f_i\in\Lief, 1\leq i\leq c,$  we have  $[[[x,f_1]_{\Lie},f_2]_{\Lie}, \dots, f_c]_{\Lie}\in \Lies\bigcap\gamma_{c+1}^{\Lie}(\Lief)=\gamma_{c+1}^{\Lie}(\Lief, \Lier).$ So $x+\gamma_{c+1}^{\Lie}(\Lief, \Lier)\in{\cal Z}_{c+1}^{\Lie}\left( \frac{\Lief}{\gamma_{c+1}^{\Lie}(\Lief, \Lier)} \right)={\Liet}/{\gamma_{c+1}^{\Lie}(\Lief, \Lier)}, $ implying that $x\in\Liet.$ Hence ${\cal Z}_{c+1}^{\Lie}( {\Lief}/{\Lies})=\Liet/\Lies.$    So   ${\cal Z}_{c+1}^{\Lie}(\Lieq^{\ast})/\Liem\cong \frac{{\cal Z}_{c+1}^{\Lie}( {\Lief}/{\Lies})}{\Lier/\Lies}\cong\frac{\Liet/\Lies}{\Lier/\Lies}\cong \Liet/\Lier,$ and  is therefore uniquely determined by the free presentation $0 \to \Lier \to \Lief \overset{\rho} \to \Lieq \to 0.$
\end{proof}

\begin{Th}
Let $\Lieq$ be a $\Lie$-nilpotent Leibniz algebra of class at most $k \geq 1$ such that ${\cal M}_{\Lie}^{(c)}(\Lieq) \neq 0$, for some $c > k$. Then $\Lieq$ has not $c$-covering.
\end{Th}
\begin{proof}
We proceed by contradiction. Let $0 \to  \Liem \to \Lieq^{\ast} \overset{\psi} \to \Lieq \to 0$ be a $c$-\Lie-stem cover  of $\Lieq.$  Then  $\Liem\subseteq\gamma^{\Lie}_{c+1}(\Lieq^{\ast})$ and $\Liem\cong{\cal M}_{\Lie}^{(c)}(\Lieq).$ Now since $\Lieq$ is $\Lie$-nilpotent of class $k,$ it follows by the proof of Corollary \ref{hopfian}  that $\gamma^{\Lie}_{k+1}(\Lieq^{\ast})\cong\Liem,$ and thus $\gamma^{\Lie}_{c+1}(\Lieq^{\ast})\subseteq\gamma^{\Lie}_{k+1}(\Lieq^{\ast})\cong\Liem$ since $c>k.$ So $\gamma^{\Lie}_{c+1}(\Lieq^{\ast})=\gamma^{\Lie}_{k+1}(\Lieq^{\ast})=\Liem.$  We claim that $\gamma^{\Lie}_{c+1}(\Lieq^{\ast})=0.$ Indeed, we have by Proposition 3.4 that $\Liem\subseteq {\ze}_c^{\Lie}({\Lieq^{\ast}}).$ Therefore $[[[\Liem,\underbrace{\Lieq^{\ast}]_{\Lie},\Lieq^{\ast}]_{\Lie},\dots,\Lieq^{\ast}}_\text{c-times}]_{\Lie}=0.$ So

 $$
 \begin{aligned}
 \gamma^{\Lie}_{2k+1}(\Lieq^{\ast}) &=[[[\gamma^{\Lie}_{k+1}(\Lieq^{\ast}),\underbrace{\Lieq^{\ast}]_{\Lie},\Lieq^{\ast}]_{\Lie},\dots,\Lieq^{\ast}}_\text{k-times}]_{\Lie}\\&=[[[\gamma^{\Lie}_{c+1}(\Lieq^{\ast}),\underbrace{\Lieq^{\ast}]_{\Lie},\Lieq^{\ast}]_{\Lie},\dots,\Lieq^{\ast}}_\text{k-times}]_{\Lie}\\&=\gamma^{\Lie}_{c+k+1}(\Lieq^{\ast})\\&=[[[\gamma^{\Lie}_{k+1}(\Lieq^{\ast}),\underbrace{\Lieq^{\ast}]_{\Lie},\Lieq^{\ast}]_{\Lie},\dots,\Lieq^{\ast}}_\text{c-times}]_{\Lie}\\&=[[[\Liem,\underbrace{\Lieq^{\ast}]_{\Lie},\Lieq^{\ast}]_{\Lie},\dots,\Lieq^{\ast}}_\text{c-times}]_{\Lie}\\&=0.
 \end{aligned}
 $$
 So if $c\geq2k,$ we have $\gamma^{\Lie}_{c+1}(\Lieq^{\ast})\subseteq\gamma^{\Lie}_{2k+1}(\Lieq^{\ast})=0.$ Otherwise, $c<2k$ and we have
 $$
 \begin{aligned}
 \gamma^{\Lie}_{3k-c+1}(\Lieq^{\ast})&=[[[\gamma^{\Lie}_{k+1}(\Lieq^{\ast}),\underbrace{\Lieq^{\ast}]_{\Lie},\Lieq^{\ast}]_{\Lie},\dots,\Lieq^{\ast}}_\text{(2k-c)-times}]_{\Lie}\\&=[[[\gamma^{\Lie}_{c+1}(\Lieq^{\ast}),\underbrace{\Lieq^{\ast}]_{\Lie},\Lieq^{\ast}]_{\Lie},\dots,\Lieq^{\ast}}_\text{(2k-c)-times}]_{\Lie}\\&=\gamma^{\Lie}_{2k+1}(\Lieq^{\ast})=0.
 \end{aligned}$$
Similarly, if $c\geq \frac{3}{2}k,$ we have $\gamma^{\Lie}_{c+1}(\Lieq^{\ast})\subseteq\gamma^{\Lie}_{3k-c+1}(\Lieq^{\ast})=0.$ Otherwise, $c< \frac{3}{2}k$ and we have
$$
 \begin{aligned}
 \gamma^{\Lie}_{4k-2c+1}(\Lieq^{\ast})&=[[[\gamma^{\Lie}_{k+1}(\Lieq^{\ast}),\underbrace{\Lieq^{\ast}]_{\Lie},\Lieq^{\ast}]_{\Lie},\dots,\Lieq^{\ast}}_\text{(3k-2c)-times}]_{\Lie}\\&=[[[\gamma^{\Lie}_{c+1}(\Lieq^{\ast}),\underbrace{\Lieq^{\ast}]_{\Lie},\Lieq^{\ast}]_{\Lie},\dots,\Lieq^{\ast}}_\text{(3k-2c)-times}]_{\Lie}\\&=\gamma^{\Lie}_{3k-c+1}(\Lieq^{\ast})=0.
 \end{aligned}$$
 This process continue and stops when $c\leq k,$  yielding  $\gamma^{\Lie}_{c+1}(\Lieq^{\ast})=0.$ This implies ${\cal M}_{\Lie}^{(c)}(\Lieq)\cong\Liem=0.$ A contradiction.
\end{proof}


\section{$c$-$\Lie$-capability} \label{capability}

In this section we introduce the $c$-$\Lie$-characteristic ideal of a Leibniz algebra, which is used to study $c$-$\Lie$-capability of Leibniz algebras.

\begin{De} \label{capable1}
A Leibniz algebra $\Lieq$ is said to be $c$-$\Lie$-capable if there exists a Leibniz algebra $\Lieh$ such that $\Lieq \cong \Lieh / {\ze}_{c}^{\Lie}(\Lieh)$.

We call $c$-$\Lie$-characteristic ideal of a Leibniz algebra $\Lieq$, denoted by ${\sf Z^{\ast}}(\Lieq)$, to the smallest two-sided ideal $\Lies$ of $\Lieq$ such that $\Lieq / \Lies$ is $c$-$\Lie$-capable.
\end{De}

\begin{Pro}
Let $\Lieq$ be a Leibniz algebra, then ${\sf Z^{\ast}}(\Lieq)$ is a two-sided ideal of $\Lieq$ contained in ${\ze}_{c}^{\Lie}(\Lieq)$, and ${\sf Z^{\ast}}\left(\Lieq/{\sf Z^{\ast}}(\Lieq) \right)=0$.
\end{Pro}
\begin{proof}
Since $\Lieq/{\ze}_{c}^{\Lie}(\Lieq)$ is $c$-$\Lie$-capable and ${\sf Z^{\ast}}(\Lieq)$ is the smallest two-sided ideal such that $\Lieq/{\sf Z^{\ast}}(\Lieq)$ is $c$-$\Lie$-capable, then ${\sf Z^{\ast}}(\Lieq) \subseteq {\ze}_{c}^{\Lie}(\Lieq)$.

  $\Lieq/{\sf Z^{\ast}}(\Lieq)$ is $c$-$\Lie$-capable by definition, then $\Lieq/{\sf Z^{\ast}}(\Lieq) \cong \Lieh/{\ze}_{c}^{\Lie}(\Lieh)$ for some Leibniz algebra $\Lieh$. Since ${\sf Z^{\ast}}\left(\Lieq/{\sf Z^{\ast}}(\Lieq) \right)$ is the smallest two-sided ideal $\Lies$ such that  $\frac{\Lieq/{\sf Z^{\ast}}(\Lieq)}{\Lies}$ is  $c$-$\Lie$-capable, this $\Lies$ should be the trivial one thanks to the above isomorphism.
\end{proof}

\begin{Pro} \label{intersection}
The two-sided ideal ${\sf Z^{\ast}}(\Lieq)$ of a Leibniz algebra $\Lieq$ is the intersection of all two-sided ideals $f \left({\ze}_{c}^{\Lie}(\Lieg) \right)$, where $f : \Lieg \twoheadrightarrow \Lieq$ is a $c$-$\Lie$-central extension.
\end{Pro}
\begin{proof}
Let $A = \bigcap \left\{ f \left( {\ze}_{c}^{\Lie}(\Lieg) \right) \mid f : \Lieg \twoheadrightarrow \Lieq \ {\rm is\ a}\  c{\rm -}\Lie{\rm - central\ extension} \right\}$ be. By Definition \ref{capable1} $\Lieq/{\sf Z^{\ast}}(\Lieq)$ is $c$-$\Lie$-capable, i.e. there exists a Leibniz algebra $\Lieh$ such that $0 \to   {\ze}_{c}^{\Lie}(\Lieh) \to \Lieh \stackrel{\theta} \to \Lieq/{\sf Z^{\ast}}(\Lieq) \to 0$. Obviously this sequence is a $c$-$\Lie$-central extension by Proposition \ref{inclusion}.

Consider the Leibniz algebra $\Liem = \left \{ (q, h) \in \Lieq \times \Lieh \mid \theta(h) = q +  {\sf Z^{\ast}}(\Lieq) \right \}$ and let $\phi : \Liem \to \Lieq$ given by $\phi(q,h) = q$. We claim that $\phi$ is a surjective homomorphism with $\ker(\phi) \subseteq {\ze}_{c}^{\Lie}(\Liem)$ and $\phi \left( {\ze}_{c}^{\Lie}(\Liem) \right) \subseteq  {\sf Z^{\ast}}(\Lieq)$.

Indeed, for any $q \in \Lieq$, consider $q +  {\sf Z^{\ast}}(\Lieq)$, then there exists $h \in \Lieh$ such that $\theta(h) = q + {\sf Z^{\ast}}(\Lieq)$. Hence $\phi(q,h)=q$.

On the other hand, $\ker(\phi) = \left \{ (0,h) \in \Lieq \times \Lieh \mid \theta(h) \in {\sf Z^{\ast}}(\Lieq) \right \}$. In order to show that   $\ker(\phi) \subseteq {\ze}_{c}^{\Lie}(\Liem)$, it is enough to show that $\gamma_{c+1}^{\Lie} \left( \Liem, \ker(\phi) \right) = 0$ thanks to Proposition \ref{inclusion}. And this fact holds because for any $(0, h) \in \ker(\phi), (q_i, h_i), 1 \leq i \leq c$, we have  $[[[(0,h),(q_1,h_1)]_{\Lie}, (q_2, h_2)]_{\Lie}, \dots, (q_c,h_c)]_{\Lie} = \left(0, [[[h,h_1]_{\Lie},  h_2]_{\Lie}, \dots, \right.$ $\left. h_c]_{\Lie} \right) \in \left( 0, \ker(\theta) \right) = \left(0, {\ze}_{c}^{\Lie}(\Lieh) \right) \subseteq {\ze}_{c}^{\Lie}(\Lieq) \times {\ze}_{c}^{\Lie}(\Lieh)  = {\ze}_{c}^{\Lie}(\Liem)$.

Finally, for any $(q, h) \in {\ze}_{c}^{\Lie}(\Liem)$, then $\phi(q,h)=q$ with $\theta(h) = q + {\sf Z^{\ast}}(\Lieq)$, but $h \in {\ze}_{c}^{\Lie}(\Lieh) = \ker(\theta)$, hence $q \in {\sf Z^{\ast}}(\Lieq)$.

Therefore we have showed that $A \subseteq {\sf Z^{\ast}}(\Lieq)$.

For the converse inclusion, consider the following commutative diagram associated to some $c$-$\Lie$-central extension $0 \to \Lien  \to \Lieg \stackrel{f} \to \Lieq \to 0$:
\[ \xymatrix{
\Lien \ \ \ar@{>->}[r] \ar@{=}[d] &  {\ze}_{c}^{\Lie}(\Lieg)  \ar@{>>}[r]  \ar@{>->}[d] & {\ze}_{c}^{\Lie}(\Lieq) \ar@{>->}[d]\\
\Lien \ \ \ar@{>->}[r] \ar@{>>}[d] & \Lieg \ar@{>>}[r]^f \ar@{>>}[d] & \Lieq \ar@{>>}[d] \\
0 \ \ \ar@{>->}[r] & \Lieg/{\ze}_{c}^{\Lie}(\Lieg) \ar[r]^{\sim} & \Lieq/{\ze}_{c}^{\Lie}(\Lieq)
} \]
where $\Lien \subseteq {\ze}_{c}^{\Lie}(\Lieg)$ by Proposition \ref{inclusion} and $f\left( {\ze}_{c}^{\Lie}(\Lieg) \right) = {\ze}_{c}^{\Lie}(\Lieq)$ by a standard induction.

Hence $\Lieq/f\left( {\ze}_{c}^{\Lie}(\Lieg) \right) \cong \Lieg/{\ze}_{c}^{\Lie}(\Lieg)$ means that $0 \to {\ze}_{c}^{\Lie}(\Lieg)
\to \Lieg \to \Lieq/f\left( {\ze}_{c}^{\Lie}(\Lieg) \right)$ $\to 0$ is a $c$-$\Lie$-central extension, i.e. $\Lieq/f\left( {\ze}_{c}^{\Lie}(\Lieg) \right)$ is $c$-$\Lie$-capable. Consequently ${\sf Z^{\ast}}(\Lieq) \subseteq f\left( {\ze}_{c}^{\Lie}(\Lieg) \right)$, since ${\sf Z^{\ast}}(\Lieq)$ is the smallest two-sided ideal satisfying this property.

So  ${\sf Z^{\ast}}(\Lieq) \subseteq f\left( {\ze}_{c}^{\Lie}(\Lieg) \right)$ for any $c$-$\Lie$-central extension $f : \Lieg \twoheadrightarrow \Lieq$, then  ${\sf Z^{\ast}}(\Lieq) \subseteq \bigcap \left\{ f\left( {\ze}_{c}^{\Lie}(\Lieg) \right) \mid f : \Lieg \twoheadrightarrow \Lieq \ {\rm is\ a}\  c{\rm -}\Lie{\rm - central\ extension} \right\}$.
\end{proof}

\begin{Le} \label{free presentation}
Let  $0 \to \Lier \to \Lief \overset{\rho} \to \Lieq \to 0$ be a free presentation and  $0 \to \Liem \to \Lieh \overset{\theta} \to \Lieq \to 0$ be  a $c$-$\Lie$-central extension of a Leibniz algebra $\Lieq$, then $\overline{\rho} \left( {\cal Z}_c^{\Lie}(\Lief / \gamma_{c+1}^{\Lie}(\Lief, \Lier)) \right) \subseteq \theta \left( {\ze}_c^{\Lie}(\Lieh) \right)$, where $\overline{\rho}$ is the natural surjective homomorphism induced by $\rho$.
\end{Le}
\begin{proof}
Since $\Lief$ is a free Leibniz algebra, then there exists a homomorphism $\beta : \Lief \to \Lieh$ such that $\theta \circ \beta = \rho$. Obviously $\beta(\Lier) \subseteq \Liem$ and $\beta \left( \gamma_{c+1}^{\Lie}(\Lief, \Lier) \right) = 0$. Then we have the following commutative diagram:
\begin{equation} \label{diagram}
 \xymatrix{
& \gamma_{c+1}^{\Lie}(\Lief, \Lier) \ar@{>->}[d] \ar@{=}[r] & \gamma_{c+1}^{\Lie}(\Lief, \Lier) \ar@{>->}[d] \ar@{>>}[r] & 0 \ar@{>->}[d]  \\
& \Lier \ \ \ar@{>>}[d] \ar@{>->}[r] & \Lief \ar@{>>}[r]^{\rho} \ar@{>>}[d]^{pr} \ar@{-->}[ddl]^{\beta}
 & \Lieq \ar@{=}[ddl] \ar@{=}[d]\\
 & \frac{\Lier}{\gamma_{c+1}^{\Lie}(\Lief, \Lier)}\ \ \ar@{>->}[r] & \frac{\Lief}{\gamma_{c+1}^{\Lie}(\Lief, \Lier)} \ar@{>>}[r]^{\ \ \ \ \ \ \overline{\rho}} \ar@{-->}[dl]^{\pi} & \Lieq \ar@{=}[dl] \\
 \Liem \ \ \ar@{>->}[r] & \Lieh \ar@{>>}[r]^{\theta} & \Lieq &
}
\end{equation}
 where $\pi$ is induced by $\beta$. Having in mind that $\Lieh = \ker(\theta) + \Im(\beta)$, then it is an easy task to verify that $\pi \left( {\ze}_c^{\Lie}(\Lief / \gamma_{c+1}^{\Lie}(\Lief, \Lier)) \right) \subseteq {\ze}_c^{\Lie}(\Lieh)$.

Since $\overline{\rho} \circ pr = \rho = \theta \circ \beta = \theta \circ \pi \circ pr$, then $\overline{\rho} \left( {\ze}_c^{\Lie}(\Lief / \gamma_{c+1}^{\Lie}(\Lief, \Lier)) \right) = \theta \left( \pi  \left( {\ze}_c^{\Lie}(\Lief / \gamma_{c+1}^{\Lie}(\Lief, \Lier)) \right) \right)$ $\subseteq \theta \left(  {\ze}_c^{\Lie}(\Lieh) \right)$.
\end{proof}

\begin{Co} \label{equality}
Let  $0 \to \Lier \to \Lief \overset{\rho} \to \Lieq \to 0$ be a free presentation of a Leibniz algebra $\Lieq$, then $\overline{\rho} \left( {\ze}_c^{\Lie}(\Lief / \gamma_{c+1}^{\Lie}(\Lief, \Lier)) \right) = {\sf Z^{\ast}}(\Lieq)$.
\end{Co}
\begin{proof}
From the $c$-$\Lie$-central extension $0 \to \frac{\Lier}{\gamma_{c+1}^{\Lie}(\Lief, \Lier)} \to \frac{\Lief}{\gamma_{c+1}^{\Lie}(\Lief, \Lier)} \overset{\overline{\rho}} \to  \Lieq \to 0$ in Example \ref{Example3.14} {\it (b)}, it follows that ${\sf Z^{\ast}}(\Lieq) \subseteq \overline{\rho} \left( {\ze}_c^{\Lie}(\Lief / \gamma_{c+1}^{\Lie}(\Lief, \Lier)) \right)$ by Proposition \ref{intersection}.

By Lemma \ref{free presentation},  $\overline{\rho} \left( {\ze}_c^{\Lie}(\Lief / \gamma_{c+1}^{\Lie}(\Lief, \Lier)) \right) \subseteq \theta \left( {\ze}_c^{\Lie}(\Lieh) \right)$, for any $c$-$\Lie$-central extension $\theta : \Lieh \twoheadrightarrow \Lieq$, then $\overline{\rho} \left( {\ze}_c^{\Lie}(\Lief / \gamma_{c+1}^{\Lie}(\Lief, \Lier)) \right) \subseteq \bigcap \theta \left( {\ze}_c^{\Lie}(\Lieh) \right) = {\sf Z^{\ast}}(\Lieq)$.
\end{proof}

\begin{Co} \label{capable}
${\sf Z^{\ast}}(\Lieq)=0$ if and only if $\Lieq$ is a $c$-$\Lie$-capable Leibniz algebra.
\end{Co}
\begin{proof}
If $\Lieq$ is a $c$-$\Lie$-capable Leibniz algebra, then there exists a $c$-$\Lie$-central extension  $0 \to {\ze}_{c}^{\Lie}(\Lieh)
\to \Lieh \stackrel{f}\to \Lieq \to 0$; now Proposition \ref{intersection} implies that ${\sf Z^{\ast}}(\Lieq) \subseteq f\left( {\ze}_{c}^{\Lie}(\Lieh) \right) = 0$.

Conversely, if ${\sf Z^{\ast}}(\Lieq)=0$, for any free presentation $0 \to \Lier \to \Lief \overset{\rho} \to \Lieq \to 0$, Corollary \ref{equality} implies that  $\overline{\rho} \left( {\ze}_c^{\Lie}(\Lief / \gamma_{c+1}^{\Lie}(\Lief, \Lier)) \right) =0$, i.e. ${\ze}_c^{\Lie}(\Lief / \gamma_{c+1}^{\Lie}(\Lief, \Lier)) \subseteq \ker(\overline{\rho}) = \Lier / \gamma_{c+1}^{\Lie}(\Lief, \Lier)$. Moreover  $0 \to \frac{\Lier}{\gamma_{c+1}^{\Lie}(\Lief, \Lier)} \to \frac{\Lief}{\gamma_{c+1}^{\Lie}(\Lief, \Lier)} \overset{\overline{\rho}} \to  \Lieq \to 0$ is a $c$-$\Lie$-central extension, then $ \Lier / \gamma_{c+1}^{\Lie}(\Lief, \Lier) \subseteq {\ze}_c^{\Lie}(\Lief / \gamma_{c+1}^{\Lie}(\Lief, \Lier))$ by Proposition \ref{inclusion}.

Thus $0 \to {\ze}_c^{\Lie}(\Lief / \gamma_{c+1}^{\Lie}(\Lief, \Lier)) \to \frac{\Lief}{\gamma_{c+1}^{\Lie}(\Lief, \Lier)} \overset{\overline{\rho}} \to  \Lieq \to 0$ is a $c$-$\Lie$-central extension, i.e. $\Lieq$ is $c$-$\Lie$-capable.
\end{proof}

\begin{Th}
 Let $0 \to \Liem \to \Lieh \overset{\psi} \to \Lieq \to 0$ be a $c$-$\Lie$-stem cover of a Leibniz algebra $\Lieq$. Then $\psi \left( {\ze}_c^{\Lie}(\Lieh) \right) = {\sf Z^{\ast}}(\Lieq)$.
\end{Th}
\begin{proof}
Let  $0 \to \Lier \to \Lief \overset{\rho} \to \Lieq \to 0$ be a free presentation of $\Lieq.$ Then by Lemma \ref{equivalence property}, there exists a two-sided ideal $\Lies$ of $\Lief$ such that  $\Lieh \cong \Lief/\Lies$ and $\Liem \cong \Lier/\Lies$, and
$\Lier/\gamma_{c+1}^{\Lie}(\Lief, \Lier) \cong {\cal M}_{\Lie}^{(c)} (\Lieq) \oplus \Lies/\gamma_{c+1}^{\Lie}(\Lief, \Lier)$.

 Now let $\Liet$ be a two-sided ideal of $\Lief$ such that  ${\cal Z}_{c}^{\Lie}( \Lieh)={\cal Z}_{c}^{\Lie}( {\Lief}/{\Lies})= {\Liet}/{\Lies}.$ We claim that  ${\cal Z}_{c}^{\Lie}\left( \frac{\Lief}{\gamma_{c+1}^{\Lie}(\Lief, \Lier)} \right )=\frac{\Liet}{\gamma_{c+1}^{\Lie}(\Lief, \Lier)}.$

  Indeed, let $x+\gamma_{c+1}^{\Lie}(\Lief, \Lier)\in {\cal Z}_{c}^{\Lie}( {\Lief}/{\gamma_{c+1}^{\Lie}(\Lief, \Lier)}).$ Then  for all $f_i\in\Lief, 1\leq i\leq c,$  we have  $[[[x,f_1]_{\Lie},f_2]_{\Lie}, \dots, f_c]_{\Lie}\in \gamma_{c+1}^{\Lie}(\Lief, \Lier)\subseteq\Lies,$ implying that  $x+\Lies\in {\cal Z}_{c}^{\Lie}( {\Lief}/{\Lies})=\Liet/ \Lies.$ So  $x\in\Liet$, and thus $x+\gamma_{c+1}^{\Lie}(\Lief, \Lier)\in{\Liet}/{\gamma_{c+1}^{\Lie}(\Lief, \Lier)}.$

 Conversely,  as ${\cal Z}_{c}^{\Lie}( {\Lief}/{\Lies})= {\Liet}/{\Lies},$ we have  for all $t\in\Liet$ and  $f_i\in\Lief, 1\leq i\leq c,$  that  $[[[t+\Lies,f_1+\Lies]_{\Lie},f_2+\Lies]_{\Lie}, \dots, f_c+\Lies]_{\Lie}= \overline{0}$, then $[[[t,f_1]_{\Lie},f_2]_{\Lie}, \dots, f_c]_{\Lie}\in\Lies.$ So $\gamma_{c+1}^{\Lie}(\Lief, \Liet)\subseteq\Lies\bigcap\gamma_{c+1}^{\Lie}(\Lief)\subseteq \gamma_{c+1}^{\Lie}(\Lief, \Lier),$ implying that $\frac{\Liet}{\gamma_{c+1}^{\Lie}(\Lief, \Lier)}\subseteq{\cal Z}_{c}^{\Lie} \left( \frac{\Lief}{\gamma_{c+1}^{\Lie}(\Lief, \Lier)} \right).$

 We now have (here we use similar notations to diagram (\ref{diagram})) $$\psi \left( {\ze}_c^{\Lie}(\Lieh) \right) = \psi \left( \Liet/\Lies \right) =\rho(\Liet)=\bar{\rho} \left( \Liet/\gamma_{c+1}^{\Lie}(\Lief, \Lier) \right)=\bar{\rho} \left( {\cal Z}_{c}^{\Lie} \left( \frac{\Lief}{\gamma_{c+1}^{\Lie}(\Lief, \Lier)} \right) \right)    ={\sf Z^{\ast}}(\Lieq)$$
  The last equality holds thanks to Corollary \ref{equality}.
\end{proof}

\begin{Th} \label{equivalent conditions}
Let $\Lien$ be a $c$-$\Lie$-central two-sided ideal of a Leibniz algebra $\Lieq$. Then the following statements are equivalent:
\begin{enumerate}
\item[(a)] $\Lien \cap \gamma_{c+1}^{\Lie}(\Lieq) \cong \frac{{\cal M}_{\Lie}^{(c)} \left(\Lieq / \Lien \right)}{{\cal M}_{\Lie}^{(c)} (\Lieq)}$.
    \item[(b)] $\Lien \subseteq {\sf Z^{\ast}}(\Lieq)$.
    \item[(c)] The natural map ${\cal M}_{\Lie}^{(c)} (\Lieq) \to {\cal M}_{\Lie}^{(c)} (\Lieq / \Lien)$ is injective.
\end{enumerate}
\end{Th}
\begin{proof}
The equivalence between statements {\it (a)} and {\it (c)} directly follows from Proposition \ref{exact sequences} {\it (c)}.

For the equivalence between statements {\it (b)} and {\it (c)}, consider the  free presentations in diagram (\ref{free present diagr}).
By exact sequence (\ref{four term}), we only need to prove that $\gamma_{c+1}^{\Lie}(\Lief, \Lies) = \gamma_{c+1}^{\Lie}(\Lief, \Lier)$ if and only if $\Lien \subseteq {\sf Z^{\ast}}(\Lieq)$.

Set $\overline{\Lief} = \frac{\Lief}{\gamma_{c+1}^{\Lie}(\Lief, \Lier)}, \overline{\Lier} = \frac{\Lier}{\gamma_{c+1}^{\Lie}(\Lief, \Lier)}, \overline{\Lies} = \frac{\Lies}{\gamma_{c+1}^{\Lie}(\Lief, \Lier)}$, then $\gamma_{c+1}^{\Lie}(\Lief, \Lies) = \gamma_{c+1}^{\Lie}(\Lief, \Lier)$ if and only if $\overline{\Lies} \subseteq  {\ze}_c^{\Lie}(\overline{\Lief})$.

By Corollary \ref{equality}, ${\sf Z^{\ast}}(\Lieq) = \overline{\rho} \left(  {\ze}_c^{\Lie}(\overline{\Lief}) \right)$, consequently, we obtain that $\overline{\rho}(\overline{\Lies}) \subseteq {\sf Z^{\ast}}(\Lieq)$ if and only if $\overline{\Lies} \subseteq {\ze}_c^{\Lie}(\overline{\Lief})$, but $\overline{\rho}(\overline{\Lies}) = \rho(\Lies) = \Lien$, which completes the proof.
\end{proof}

\begin{Co}
Any Leibniz algebra $\Lieq$ is $c$-$\Lie$-capable if and only if the natural map ${\cal M}_{\Lie}^{(c)} (\Lieq) \to {\cal M}_{\Lie}^{(c)} (\Lieq / \langle x \rangle)$ has non-trivial kernel for all non-zero element $x \in {\cal Z}_c^{\Lie}(\Lieq)$.
\end{Co}
\begin{proof}
Assume that $\ker \left(\sigma_x : {\cal M}_{\Lie}^{(c)} (\Lieq) \to {\cal M}_{\Lie}^{(c)} (\Lieq / \langle x \rangle)\right) = 0$  for any non-zero element $x \in Z_c^{\Lie}(\Lieq)$. By Theorem \ref{equivalent conditions}, $\sigma_x$ injective if and only if  $\langle x \rangle \subseteq Z^{\ast}(\Lieq)$, so $Z^{\ast}(\Lieq) \neq 0$, i.e. $\Lieq$ is not $c$-$\Lie$-capable  by Corollary \ref{capable}.

For every non-zero element $x \in Z_c^{\Lie}(\Lieq)$, we have $0 \neq \langle x \rangle \nsubseteqq Z^{\ast}(\Lieq) = 0$, then $\sigma_x$ cannot be an injective homomorphism.
\end{proof}

\begin{Pro}  \label{prop 10}
Let $\Lien$ be a $c$-$\Lie$-central two-sided ideal of a Leibniz algebra $\Lieq$. Then $\Lien \subseteq {\sf Z^{\ast}}(\Lieq)$ if and only if the natural surjection  $\Lieq \twoheadrightarrow \Lieq / \Lien$ induces an isomorphism $\gamma_{c+1}^{\Lie \ast}(\Lieq) \overset{\cong} \to \gamma_{c+1}^{\Lie \ast}(\Lieq / \Lien)$.
\end{Pro}
\begin{proof}
By the proof of Theorem \ref{equivalent conditions},  $\gamma_{c+1}^{\Lie}(\Lief, \Lies) = \gamma_{c+1}^{\Lie}(\Lief, \Lier)$ if and only if $\Lien \subseteq {\sf Z^{\ast}}(\Lieq)$.

From  diagram  (\ref{free present diagr}),  the kernel of the induced surjective homomorphism $\frac{\gamma_{c+1}^{\Lie}(\Lief)}{\gamma_{c+1}^{\Lie}(\Lief, \Lier)} = \gamma_{c+1}^{\Lie \ast}(\Lieq)  \twoheadrightarrow \gamma_{c+1}^{\Lie \ast}(\Lieq / \Lien) = \frac{\gamma_{c+1}^{\Lie}(\Lief)}{\gamma_{c+1}^{\Lie}(\Lief, \Lies)} $ is $\frac{\gamma_{c+1}^{\Lie}(\Lief, \Lies)}{\gamma_{c+1}^{\Lie}(\Lief, \Lier)}$, thus the proposition is obvious.
\end{proof}

\begin{Co} \label{formula dimension}
Let $\Lien$ be a two-sided ideal of a finite-dimensional Leibniz algebra $\Lieq$ such that $\Lien \subseteq {\cal Z}_{c}^{\Lie}(\Lieq)$. Then $\Lien \subseteq {\sf Z^{\ast}}(\Lieq)$ if and only ${\rm dim} \left({\cal M}_{\Lie}^{(c)} (\Lieq / \Lien) \right) = {\rm dim} \left( {\cal M}_{\Lie}^{(c)} (\Lieq) \right) + {\rm dim} \left( \Lien \cap \gamma_{c+1}^{\Lie}(\Lieq) \right)$.
\end{Co}
\begin{proof}
By the proof of Theorem \ref{equivalent conditions} and Corollary \ref{several properties} {\it (d)}.
\end{proof}


\section*{\bf Acknowledgements}
Second author was supported by Agencia Estatal de Investigación (Spain), grant MTM2016-79661-P (AEI/FEDER, UE, support included).


\begin{center}

\end{center}

\end{document}